\documentclass[reqno]{amsart} %REQNO = For left-aligned equation enumeration.
\usepackage[headings]{fullpage}
\usepackage{amsmath} 
\usepackage{paralist}
\usepackage{amssymb} 
\usepackage{amsthm}
\usepackage{amsfonts}
\usepackage[centertags]{amsmath}
\usepackage{fullpage}
\usepackage{comment} 
\usepackage{bm}
\usepackage{thmtools}
\usepackage{thm-restate}
\usepackage{hyperref, float}
\hypersetup{colorlinks=true,citecolor=purple,linkcolor=purple}

\usepackage[backrefs]{amsrefs}

\usepackage{tikz-cd}
\usetikzlibrary{positioning, calc}
\usepackage{caption}
\usepackage{subcaption}
\usepackage{graphicx}
\usepackage{mathrsfs}
\usepackage[export]{adjustbox} %To add frames to figures.
\usepackage{xcolor}
\usepackage[margin=1.25in]{geometry}
\usepackage{amscd}
\usetikzlibrary{shapes}
\usetikzlibrary{arrows}
\usetikzlibrary{positioning}

\tikzstyle{main} = [draw, circle, draw=black, fill=black, inner sep = 0pt, minimum size = 6pt]

\theoremstyle{definition}
\newtheorem{theorem}{Theorem}[section]

\newtheorem{lemma}[theorem]{Lemma}

\newtheorem{definition}[theorem]{Definition}
\newtheorem{corollary}[theorem]{Corollary}

\theoremstyle{definition}
\newtheorem*{thm*}{Theorem}

\newtheorem{prop}[theorem]{Proposition}

\newtheorem{example}[theorem]{Example}

\newtheorem{remark}[theorem]{Remark}

\newcommand{\C}{\mathbb{C}}

\newcommand{\Q}{\mathbb{Q}}

\newcommand{\ass}{\text{Ass}}

\newcommand{\Min}{\text{Min}}

\DeclareMathOperator{\Spec}{Spec}
\DeclareMathOperator{\spec}{Spec}
\DeclareMathOperator{\height}{ht}

\begin{document}

\title[Dimension-Preserving Saturated Embeddings]{Dimension-Preserving Saturated Embeddings of Finite Posets into the Spectra of Noetherian UFDs}
\author{David Baron and S. Loepp}

%\maketitle

\begin{abstract}
Given a finite poset $X$, we find necessary and sufficient conditions for there to exist a local Noetherian UFD $A$ and a saturated embedding of posets $\phi : X \longrightarrow \spec(A)$ such that $\dim(X)=\dim(A)$. 
The conditions imposed on \( X \) in our characterization are remarkably mild, demonstrating that there is a large class of finite posets that can be embedded into the spectrum of a local Noetherian UFD of the same dimension as $X$ in a way that preserves saturated chains.
%As a consequence, we characterize dimension-preserving saturated embeddings of finite posets into the prime spectra of Noetherian UFDs. 
We also show that given any finite poset $Y$, there exists a semi-local quasi-excellent ring $S$ and a saturated embedding $\psi: Y \longrightarrow \spec(S)$ such that if $z$ is a minimal element of $Y$, then $\psi(z)$ is a minimal prime ideal of $S$ and the coheight of $\psi(z)$ is the same as the length of the longest chain in $Y$ that starts at $z$ and ends at a maximal element of $Y$. 
\end{abstract}

\maketitle

\section{Introduction}

A central goal in commutative algebra is to understand the structure of the prime spectrum of a ring. Given a commutative ring \( R \), its prime spectrum, denoted by \(\operatorname{Spec}(R)\), is the set of all prime ideals of \( R \). There are several approaches to studying this set.  In this paper, we regard \(\operatorname{Spec}(R)\) as a partially ordered set (poset) under inclusion, a perspective that has been extensively explored (for a nice survey on the topic, see \cite{wiegand}). One of the central open questions is to find conditions for when a given poset \( X \) is isomorphic, as a poset, to the spectrum of a commutative ring $R$ where $R$ has a particular algebraic property. In \cite{Hochster}, Hochster finds necessary and sufficient conditions for a poset to be isomorphic to the prime spectrum of a commutative ring. This result is striking and prompts further inquiry, particularly within the realm of Noetherian rings. Specifically, it is natural to ask which posets can be realized as the prime spectrum of a Noetherian ring. This question remains open even in the case where the dimension of the poset is two.
Given the difficulty of this problem, a more tractable version is to focus on finite partially ordered sets. Specifically, for a given finite poset \( X \), one may ask whether there exists a Noetherian ring \( R \) such that \( X \) can be embedded into a subposet of \(\operatorname{Spec}(R)\). It is important to note that when \(\dim(X) \ge 2\), any embedding of a finite poset $X$ into \(\operatorname{Spec}(R)\) must be strict. Indeed, a well-known property of Noetherian rings is that if \( P_0 \subsetneq P_1 \subsetneq P_2 \) are prime ideals in \(R\), then there exist infinitely many prime ideals \( Q \) in $R$ with \( P_0 \subsetneq Q \subsetneq P_2 \).

In 1979, Heitmann \cite{HeitmannNoncatenary} achieved a remarkable breakthrough by proving that any finite partially ordered set can be embedded into the prime spectrum of some Noetherian ring in a way that saturated chains are preserved. Heitmann's discovery not only addresses the finite case in a definitive manner but also spurred further research into the poset structures present in \(\operatorname{Spec}(R)\) for rings with desirable algebraic or geometric properties. The concept of a map between posets that preserves saturated chains proves very useful in this context. To formalize this kind of map, we use the notion of saturated embeddings as defined in \cite{Colbert}.
%We note that a saturated embedding is a stronger condition than an embedding. In poset embeddings, extra inclusion relations can appear when prime ideals from the ring, outside the image of the poset map, introduce unexpected orderings. 
What makes saturated embeddings remarkable is that they not only preserve saturated chains from the poset in the prime spectrum of the ring, but they also ensure that no new order relations arise among the images of elements in the poset. In other words, the ordering of prime ideals in the image reflects exactly the ordering in the original poset. 
%What makes saturated embeddings remarkable is that they preserve saturated chains from the poset in the prime spectrum of the ring. 
We provide a rigorous definition of a saturated embedding in Section $\ref{posetprelims}$. 

Saturated embeddings can be used to construct rings that are not catenary. Recall that a ring is called a catenary ring if every saturated chain of prime ideals between any two fixed prime ideals has the same length. This property is fundamentally linked to the poset structure of $\spec(R)$ since the arrangement of prime ideals reflects the chain conditions within the ring. Early in the twentieth century, it was widely believed that Noetherian integral domains were necessarily catenary; however, Nagata's construction in 1956 \cite{Nagata} of a family of noncatenary Noetherian local domains provided a striking counterexample. In a similar vein, Ogoma \cite{Ogoma} and Heitmann \cite{HeitmannUFD} constructed, respectively, a noncatenary integrally closed domain and a noncatenary unique factorization domain (UFD). These developments naturally lead to the question: How noncatenary can a ring with a ``nice property" be? Investigating the possible saturated embeddings of a finite poset into the prime spectrum of a ring with such properties offers a promising avenue toward answering this question. For instance, complete local rings and excellent rings are catenary, thereby severely limiting the possible poset structures found within their prime spectra. In contrast, unique factorization domains, often regarded as well-behaved rings, are necessarily catenary up to dimension three but may fail to be catenary in higher dimensions. This observation motivates the following question: Given a finite poset \( X \), when does there exist a saturated embedding from \( X \) into \(\operatorname{Spec}(A)\) where \( A \) is a Noetherian UFD? If such an embedding exists, then by choosing $X$ to be noncatenary, the saturated embedding property would ensure that $A$ itself is noncatenary. 

In \cite{Colbert}, the authors provide a definitive answer, proving that if $X$ is \emph{any} finite poset, then there is a Noetherian UFD $A$ and a saturated embedding from $X$ into $\Spec(A)$. The key idea behind their proof is to embed the poset $X$ into the spectrum of a Noetherian UFD $A$ with sufficient depth, yielding a Noetherian UFD $A$ whose Krull dimension is approximately twice that of the dimension of $X$. While the result in \cite{Colbert} is striking from an order-theoretic perspective, a more complete understanding demands that we also address the relationship between the dimension of the finite poset and the dimension of the Noetherian UFD. In particular, a natural question is, given a finite poset $X$, when is there a local (Noetherian) UFD $A$ such that there is a saturated embedding from $X$ into $\Spec(A)$ and such that the dimension of $X$ is the same as the Krull dimension of $A$?
In the main theorem of this article, we fully characterize when this is possible. Specifically, in Section \ref{characterization of dimension-preserving saturated embeddings of UFDs} we present the following result.

 {\bf Main Theorem} (Theorem \ref{main theorem of characterization of embeddings of UFDs}) Let $X$ be a finite poset. Then, there exists a local (Noetherian) unique factorization domain $A$ and a saturated embedding $\phi: X \longrightarrow \spec(A)$ with dim$(X) = \dim(A)$ if and only if:
    \begin{enumerate}
        \item $X$ has a unique minimal element and a unique maximal element, and
        \item If $\dim(X) \geq 2$ and $x, y  \in X$ are such that $\height(x)=1$, $x< y,$ and $z \in X$ with $x \leq z \leq y$ implies that $z = x$ or $z = y$, then $\height(y)=2$.
    \end{enumerate}

\noindent The conditions imposed on \( X \) in our characterization are extremely mild, demonstrating that almost all finite posets with a unique minimal element and a unique maximal element, even those that are very noncatenary, can be embedded into the spectrum of a local (Noetherian) UFD of the same dimension as $X$ in a way that preserves saturated chains. After stating and proving our main theorem, we present two corollaries.  The first, Corollary \ref{corollary UFD for non local posets}, provides necessary and sufficient conditions on a finite poset $X$ for there to exist a (not necessarily local) Noetherian UFD $A$ such that there is a dimension-preserving saturated embedding from $X$ to $\Spec(A)$. In the second corollary (Corollary \ref{dimplusone}), we show that if $X$ is any finite poset, then there exists a Noetherian UFD $A$ and a saturated embedding from $X$ to $\Spec(A)$ where the dimension of $A$ is one more than the dimension of $X$.
%can very noncatenary. Moreover, by constructing posets and rings with matching dimensions, our work takes an important step toward resolving the broader question of when a poset \( X \) is isomorphic to the entire spectrum \(\operatorname{Spec}(R)\) of a ring \( R \) with interesting algebraic and geometric properties.

Showing that the conditions stated in our main theorem are necessary is straightforward. Thus, the bulk of our work is devoted to establishing that the conditions are indeed sufficient. In the cases where the poset \(X\) has dimension zero or one, the desired embeddings can be easily constructed. Hence, our main focus lies on the more challenging scenario when \(\dim(X) \ge 2\). In this context, when $X$ satisfies the necessary conditions, we construct a local (Noetherian) UFD \(A\) and a saturated poset embedding
\(\psi: X \longrightarrow \operatorname{Spec}(A)
\) that preserves dimension.

Our approach is divided into three main parts. Let \(X\) be a finite poset of dimension \(n \geq 2\) satisfying conditions one and two of our main theorem. We define \(X^{n-1}\) to be the subposet of $X$ consisting of all elements of $X$ with height at least one, endowed with the order inherited from \(X\). The first step is to extend results from \cite{Colbert} to obtain a reduced quasi-excellent local ring \(B'\) and a dimension-preserving saturated embedding \(\phi': X^{n-1} \longrightarrow \operatorname{Spec}(B').\) Specifically, we show in Theorem $\ref{local quasi-excellent ring codimension preserving theorem}$ that for any finite poset $Y$ with a unique maximal node, there exists a reduced quasi-excellent local ring $S$ and a $\emph{coheight-preserving}$ saturated embedding $\Psi: Y \longrightarrow \spec(S)$ along the minimal prime ideals of $S$. A coheight-preserving map is a stronger notion than a dimension-preserving map, which we explain in detail in Section $\ref{quasi-excellent ring construction}$.

The second key step is to construct a local ring $B$ from $B'$ so that there exists a saturated embedding $\psi: X \longrightarrow \spec(B)$ (Theorem~$\ref{local domain with desired properties}$). The main challenge at this stage is to construct $B$ with the necessary algebraic and geometric properties to allow the construction of a local UFD $A \subseteq B$, such that $A$ and $B$ share the same completion. The last step is to exploit properties of $B$ and use tools from \cite{Colbert}, \cite{HeitmannUFD} and \cite{Bonat} to construct a local UFD $A$ satisfying certain properties. Specifically, for any element \(x \in X\) of height two or greater, we verify that the image \(\psi(x)\) satisfies certain conditions that allow us to construct $A$ so that it contains a generating set for $\psi(x)$. For elements of $X$ of height one, we adapt analogous techniques from \cite{Bonat} to carefully select elements from specific height one prime ideals of \(B\) to adjoin to $A$, ensuring that the inclusion relations inherent in \(B\) are preserved in \(A\). These technical details are elaborated in the subsequent sections. Finally, we show that \(A\) is the desired UFD and that there exists a dimension-preserving saturated embedding $\phi$ from \(X\) into \(\operatorname{Spec}(A)\). 

In Section~\ref{posetprelims}, we cover essential background material on partially ordered sets. In Sections~\ref{Gluing and Retraction Preliminaries} and~\ref{Gluing and Retraction in Rings} we introduce two operations on posets from \cite{Colbert} (along with their ring-theoretic analogues) that are central to establishing that there exists a coheight-preserving saturated embedding from any poset with a unique maximal element into the prime spectrum of a quasi-excellent local ring (Theorem~\ref{local quasi-excellent ring codimension preserving theorem}). In Section~\ref{local domain construction}, we construct a local domain with specific geometric properties and apply this result in Section~\ref{ufd construction} to construct a unique factorization domain and verify the sufficiency of the conditions in Theorem~\ref{main theorem of characterization of embeddings of UFDs}. Finally, in Section~\ref{characterization of dimension-preserving saturated embeddings of UFDs}, we prove Theorem~\ref{main theorem of characterization of embeddings of UFDs}, thus fully characterizing when a finite poset \(X\) admits a dimension-preserving saturated embedding \(\psi: X \longrightarrow \spec(A)\) for some local UFD \(A\).

\vspace{.2cm}

\noindent {\bf Notation and Conventions}. All rings in this paper are assumed to be commutative with unity. We use the standard convention of calling a partially ordered set a poset. If $X$ is a poset and $x$ is an element of $X$, we will sometimes say that $x$ is a node of $X$. As a convention for this paper, when we say that $X$ is a poset, it will be assumed that $X$ is not empty. When we write that for a set $Y$, $|Y| = c$, we mean that the cardinality of $Y$ is equal to the cardinality of $\mathbb{R}$.
We use $\min X$ to denote the set of minimal elements of a poset $X$ and $\Min(R)$ to denote the set of minimal prime ideals of the ring $R$. When we say that a ring $R$ is quasi-local, we mean that $R$ has exactly one maximal ideal and need not be Noetherian.  We use $(R,M)$ to denote a quasi-local ring with maximal ideal $M$. A local ring is a quasi-local ring that is Noetherian.  If $(R,M)$ is a local ring, we use $\widehat{R}$ to denote the $M$-adic completion of $R$. Finally, if $R$ is a ring and $P$ is a prime ideal of $R$, we define the coheight of $P$ to be the Krull dimension of the ring $R/P$.

\section{Poset Preliminaries}\label{posetprelims}

In this section, we introduce basic definitions related to partially ordered sets and we define the notion of a dimension-preserving saturated embedding. 
%\begin{definition}[Partially Ordered Set]
%Let \(X\) be a set and let \(\leq\) be a binary relation on \(X\). We say that the relation \(\leq\) is a partial order if it satisfies:
%\begin{enumerate}
%  \item \emph{Reflexivity} if for every \(x \in X\), we have \(x \leq x\).
%  \item \emph{Antisymmetry} if for all \(x, y \in X\), \(x \leq y\) and \(y \leq x\) imply \(x = y\).
%  \item \emph{Transitivity} if for all \(x, y, z \in X\), \(x \leq y\) and \(y \leq z\) imply \(x \leq z\).
%\end{enumerate}
%A \emph{partially ordered set} (poset) is a pair \((X, \leq)\) where \(\leq\) is a partial order.
%\end{definition}
\begin{definition}
    Let $X$ be a poset, and let $x,y \in X$. We say $y$ \emph{covers} $x$, and we write $x <_c y$, if $x < y$ and for all $z\in X$, if $x\leq z\leq y$, then $z=x$ or $z=y$.
\end{definition}

Given a poset $X$ and two elements $x,y \in X$ with $x < y$, we define what it means for a chain that starts at $x$ and ends at $y$ to be saturated.  Note that our definition is analogous to the definition of a saturated chain of prime ideals in a ring.

\begin{definition}
Let $X$ be a poset and let $x,y \in X$ with $x < y$.  We say that the chain $$x = x_0 < x_1 < \cdots < x_k = y$$ is \emph{saturated} if $x_{i}$ covers $x_{i-1}$ for every $i = 1,2 \ldots, k$.
\end{definition}

%Analogous to the ring setting, we define a \emph{saturated chain} in a finite poset \(X\) as follows: given \(x, y \in X\), the chain \(x \leq y\) is \emph{saturated} if \(y\) \emph{covers} \(x\). The \emph{height} of a node \(x  \in X\), denoted \(\operatorname{ht}(x)\), is defined as the length of the longest saturated chain that begins at a minimal element \(y\) of \(X\) and ends at \(x\). 

%With this notion in place, we now introduce a particular subposet that plays a central role in the construction of the UFD described in Theorem~\ref{main theorem of characterization of embeddings of UFDs}.

%\begin{definition}[$n-1$ skeleton of a poset]\label{n-1 skeleton}
%    Let $X$ be a poset of dimension $n \geq 2$. The $n-1$ skeleton, denoted as $X^{n-1}$, is the induced poset of $X$ given by $$X^{n-1}=\{x \in X: \height(x) \geq 1 \}$$
%    where $x \leq y$ in $X^{n-1}$ if and only if $x \leq y$ in $X$.
%\end{definition}

\begin{definition}
    Let $X$ be a poset and let $C$ be a nonempty finite chain in $X$. The \emph{length} of $C$, denoted $L(C)$, is defined to be $|C|-1$
    where $|C|$ means the cardinality of $C$. 
\end{definition}

We next define the height of an element of a poset and the dimension of a poset.  Note that these definitions are analogous to the definition of the height of a prime ideal of a ring and the definition of the dimension of a ring.  

\begin{definition}
Let $X$ be a poset and let $x \in X$. The \emph{height} of $x$, denoted \(\operatorname{ht}(x)\), is defined as the length of the longest saturated chain that begins at a minimal element \(y\) of \(X\) and ends at \(x\).
\end{definition}

%As with the dimension of a Noetherian ring, there are several equivalent ways to define the dimension of a poset. We adopt the following definition.

\begin{definition}
The \emph{dimension} of a nonempty poset \(X\), denoted $\dim(X)$, is defined as
\[ \dim(X) = 
\sup\{L(C)\mid  C \ \text{is a finite chain of elements in } X\}.
\]
\end{definition}

We end this section by defining several kinds of maps between posets.

\begin{definition} Let $X$ and $Y$ be posets. We say $f : X \longrightarrow Y$ is a \emph{poset map} if for all $x,y \in  X, x \leq y$ implies $f(x) \leq f(y)$. We say a poset map $f : X \longrightarrow Y$ is a \emph{poset embedding} if for all $x,y \in X, f(x) \leq f(y)$ implies $x \leq y$. A surjective poset embedding from $X$ onto $Y$ is called a \emph{poset isomorphism}.
\end{definition}
Suppose that $f : X \longrightarrow Y$ is a poset embedding and let $x,y \in X$ with $f(x) = f(y)$. Then $x \leq y$ and $y \leq x$ and so $x = y$. 
Hence, poset embeddings are injective, and so a surjective poset embedding from $X$ onto $Y$ is bijective. It therefore makes sense to call it a poset isomorphism.
%We conclude the section by defining a stronger version of a poset embedding and an essential lemma. 
\begin{definition}[Definition 2.3, \cite{Colbert}]
Let $X$ be a poset, and let $x,y \in X$. We say $f:X\longrightarrow Y$ is a \emph{saturated embedding} if $f$ is a poset embedding and for all $x,y \in X$, if $y$ covers $x$, then $f(y)$ covers $f(x)$ in $Y$. If $Z \subseteq Y$ is a subposet (i.e. $Z$ is a poset under the same order relation on $Y$) we say $Z$ is a \emph{saturated subset} of $Y$ if for all $u,v\in Z$, whenever $v$ covers $u$ in $Z,$ it also covers $u$ in $Y$.
\end{definition}

It is not difficult to show (see Remark 2.4 in \cite{Colbert}) that if $f : X \longrightarrow Y$ is a saturated embedding, then $y$ covers $x$ in $X$ if and only if $f(y)$ covers $f(x)$ in $Y$. Note that saturated embeddings preserve saturated chains.  That is, if $f : X \longrightarrow Y$ is a saturated embedding and $C$ is a saturated chain of $X$, then $f(C)$ is a saturated chain of $Y$.

%The following lemma shows that the existence of a saturated embedding  $f:X\longrightarrow Y$ is equivalent to $X$ being isomorphic to a saturated subset of $Y$.

The next lemma shows that if $f : X \longrightarrow Y$ is a saturated embedding then $X$ is isomorphic to a saturated subset of $Y$.

\begin{lemma}[Lemma 2.7, \cite{Colbert}]\label{poset isomorphism from saturated embedding}
    If $f:X\longrightarrow Y$ is a saturated embedding of posets, then $f$ is a poset isomorphism from $X$ onto $f(X)$, and $f(X)$ is a saturated subset of $Y$.
\end{lemma}

Finally, we define the type of poset map that is the main focus of this article.

\begin{definition}
    We say a saturated embedding of posets $f: X \longrightarrow Y$ is \emph{dimension-preserving} if $\dim(X)=\dim(Y)$.
\end{definition}

Let $X$ be a finite poset. In Theorem \ref{main theorem of characterization of embeddings of UFDs}, we find necessary and sufficient conditions for there to exist a local UFD $A$ and a dimension-preserving saturated embedding of posets $\phi : X \longrightarrow \spec(A)$. By Lemma \ref{poset isomorphism from saturated embedding}, $\phi$ is a poset isomorphism from $X$ onto $\phi(X)$ and $\phi(X)$ is a saturated subset of $\Spec(A)$. Informally, this means that we have characterized exactly when there is a local UFD $A$ with the same dimension as $X$ such that $X$ occurs as part of $\spec(A)$ when $\spec(A)$ is viewed as a partially ordered set with respect to inclusion. Along the way, we also show that if $Y$ is a finite poset with a unique maximal node $y$, then there is a quasi-excellent local ring $S$ and a saturated embedding $\psi : Y \longrightarrow \spec(S)$ such that if $z$ is a minimal element of $Y$, then $\psi(z)$ is a minimal prime ideal of $S$ and the coheight of $\psi(z)$ is the same as the length of the longest chain in $Y$ that starts at $z$ and ends at $y$ (see Theorem \ref{local quasi-excellent ring codimension preserving theorem}). 

\section{Gluing and Retraction Preliminaries}\label{Gluing and Retraction Preliminaries}

In this section, we introduce important concepts and results from \cite{Colbert} that we ultimately use to prove Theorem \ref{main theorem of characterization of embeddings of UFDs}. In particular, we discuss in detail two poset operations introduced in Section 5 of \cite{Colbert}: gluing and retraction. These operations and the results associated with them will play a central role throughout this paper. We begin by defining the gluing operation. To do so, we must first define compatible maps and complete subsets.
%We also include results from \cite{Colbert} relating to these important operations.

%Our objective in this section is to lay the foundations to prove two of our main results: given a finite poset $Y$ with a unique maximal node, there exists reduced quasi-excellent ring $S$ and a codimension-preserving saturated embedding $\Psi: Y \longrightarrow \spec(S)$ (Theorem $\ref{local quasi-excellent ring codimension preserving theorem}$) and, given a finite poset $X$ satisfying mild conditions, there exists a local quasi-excellent domain $B$ satisfying specific geometric properties such that $X$ is isomorphic to a saturated subset of the prime spectrum of $B$ and $\dim(X)=\dim(B)$ (Theorem $\ref{local domain with desired properties}$). As discussed in the introduction, these theorems are fundamental to our three-step approach for proving that the conditions given in Theorem~$\ref{main theorem of characterization of embeddings of UFDs}$ are indeed sufficient.

%\subsection{Poset Preliminaries}\label{poset preliminaries}

%We discuss in detail two poset operations introduced in Section 5 of \cite{Colbert}: gluing and retracting. These operations play a central role throughout the paper. We conclude this section with a remarkable theorem from Colbert and Loepp in \cite{Colbert}.

%\subsubsection{Gluing}
\begin{definition}[Definition 5.3, \cite{Colbert}]
    Let $X$ and $Y$ and be posets, and let $g : X \longrightarrow Y $ be a surjective poset map. If $Z$ is a poset and $h : X \longrightarrow Z$ is any poset map, we say $h$ is \emph{compatible} with $g$ if, for each $y \in Y$, the restriction of $h$ to $g^{-1}(y)$ is constant.
\end{definition}
\begin{definition}
[Definition 2.5, \cite{Colbert}]     If $C\subseteq Y$ is a subposet of $Y$, we say that $C$ is a \emph{complete subset} of $Y$ if for all $u,v\in C$, if $u \leq y \leq v$ for some $y \in Y$, then $y \in C$.
\end{definition}
\begin{definition}[Definition 5.4, \cite{Colbert}]
    Let $X$ be a poset. We say that the poset $Y$ is a \emph{gluing} of $X$ with gluing map $g:X
    \longrightarrow Y$ if whenever $h:X\longrightarrow Z$ is compatible with $g$, there exists a unique poset map $\varphi:Y\longrightarrow Z$ such that $\varphi \circ g = h$. Moreover, if $C \subseteq X$ is a complete subset of $X$ such that:
    \begin{enumerate}[(i)]
        \item $g$ is constant on $C$, and
        \item if $g(x) = g(x')$ for distinct $x,x' \in X$, then both $x,x' \in C$,
    \end{enumerate}
    then we say $Y$ is a gluing of $X$ along $C.$
\end{definition}
The following theorem from \cite{Colbert} shows that there is a gluing of a poset along any given complete subset. Moreover, the theorem provides a useful property of gluings along the given complete subset. 

 \begin{theorem}[Theorem 5.7, \cite{Colbert}]\label{theorem about order induced by gluings}
 If $(X,\leq)$ is a poset and $S$ is a complete subset of $X$, then there is a gluing $Y$ of $X$ along $S$. Moreover, if $Z$ is a gluing of $X$ along $S$ via any gluing map $h : X \longrightarrow Z$, then whenever $h(x)\leq_Z h(y)$, either $x\leq_X y$ or there exists $s,t \in S$ such that $x\leq_Xs$ and $t\leq_X y$.
 \end{theorem}

As is the case in \cite{Colbert}, we will primarily focus on gluing along a subset of minimal nodes of our poset. The following definitions and lemmas make this idea precise and show useful properties of height zero gluings. Note that for a poset $X$, any subset of $\min X$ is a complete subset of $X$.
\begin{definition}[Definition 5.9, \cite{Colbert}]
    If $Y$ is a gluing of $X$ along a subset $C \subseteq \min X$, we say that $Y$ is a \emph{height zero gluing} of $X.$
\end{definition}
\begin{lemma}[Lemma 5.10, \cite{Colbert}]\label{gluing preserves dimension}
    If $Y$ is a height zero gluing of a finite poset $X$ along $C$ with gluing map $g$, then the following statements are true:
    \begin{enumerate}
        \item $g(\min X)=\min Y$,
        \item $\dim(X)=\dim(Y)$, and
        \item if $g(x) <_c g(y)$, there exists $x' \in X$ such that $x' <_c y$ and $g(x') = g(x)$.
    \end{enumerate}
\end{lemma}
In Section $\ref{quasi-excellent ring construction}$, we introduce a stronger notion of dimension-preserving saturated embeddings, which we call coheight-preserving saturated embeddings. The idea behind this concept is the following: given a Noetherian ring $R$, there are finitely many minimal prime ideals $\{P_1,\dots,P_n\}$ of $R$ and so we can consider the coheight of $P_i$ for each $i=1,\dots,n$. We translate this idea to posets and we then introduce poset maps that preserve coheight at the minimal nodes. The following definitions are essential for making this coheight-preserving property rigorous, as they provide a poset analogue of the coheight of a minimal prime ideal of a ring.
\begin{definition}
    Let $X$ be a poset of dimension $n \geq 2$, and let $x \in X$. The \emph{up set} of $x$, denoted as $x_X^{\uparrow}$, is the induced poset of $X$ given by $$x_X^{\uparrow}=\{y \in X \mid x \leq y \}$$
    where $y \leq z$ in $x_X^{\uparrow}$ if and only if $y \leq z$ in $X$. When the underlying poset is clear, we simply write $x^{\uparrow}$ to denote the up set of $x$ in $X$.
\end{definition}
\begin{lemma}\label{dimension of gluing poset}
    Let $X$ be a finite poset with a unique maximal node. Suppose $Y$ is a height zero gluing of $X$ along $C$ with gluing map $g$. Let $x_0$ be a minimal node of $X$. Then, $\dim(g(x_0)^{\uparrow}_Y)=\text{max}\{\dim(x_X^{\uparrow})\mid g(x)=g(x_0)\}$
\end{lemma}
\begin{proof}
    Suppose $\dim(g(x_0)_Y^\uparrow)=m$ and $\text{max}\{\dim(x_X^{\uparrow})\mid g(x)=g(x_0)\}=n$. Since $g$ is surjective, there are elements $x_1, \ldots ,x_m$ of $X$ such that $g(x_0)<_cg(x_1)<_c\dots<_c g(x_m)=g(x_M)$ where $x_M$ is the maximal node of $X$. Since $g$ is a height-zero gluing, $x_i<_c x_{i+1}$ for all $i=1,\dots,m-1$. Thus, we focus on $g(x_0)<_cg(x_1)$. By Lemma $\ref{gluing preserves dimension}$, there exists $\mathbf{x} \in X$ such that $\mathbf{x}<_c x_1$ and $g(\mathbf{x})=g(x_0)$. So, $\mathbf{x}<_c x_1<_c\dots<_c x_M$ is a saturated chain of length $m$ in $X$. Hence, $\dim(\mathbf{x}_X^\uparrow) \geq m$ and so $n \geq m$. 

    Now, assume that $\mathbf{x}$ is a minimal node of $X$ such that $g(\mathbf{x})=g(x_0)$ and $\dim(\mathbf{x}_X^\uparrow)=n$. Let $\mathbf{x}<_c x_1<_c\dots<_c x_n= x_M$ be a saturated chain of length $n$ in $X$. We claim that its image in $Y$ under $g$ has length at least $n$. Since $g$ is a poset map, $g(\mathbf{x}) \leq g(x_1) \leq \dots \leq g(x_n) = g(x_M).$ If for some $i = 1,\dots, n-1$, $g(x_i)= g(x_{i+1})$, then by the properties of gluing maps, $x_i,x_{i+1} \in C \subseteq \min X$, which is a contradiction. By the same argument, $g(\mathbf{x}) = g(x_1)$ implies $x_1 \in \min X$, which is also a contradiction. Thus, $g(x_0) = g(\mathbf{x}) < g(x_1) < \dots < g(x_n) = g(x_M).$ Hence, $n \leq m$ and it follows that
\[
g(x_0)_Y^\uparrow= \max\{\dim(x_X^{\uparrow})\mid g(x)=g(x_0)\}.\] \end{proof}

We now turn our focus to the operation of retraction.

%\subsubsection{Retraction}
    \begin{definition}[Definition 5.14, \cite{Colbert}]\label{retraction definition}Let $X$ be a poset such that dim$(X) > 0$. We say a height one node $x$ is \emph{simple} if for all $u \in X$ such that $u <x$, $x$ is the only node that covers $u$. If $x$ is a simple node in $X,$ define $D_x:=\{u\in X:u <x\}$. We say $X'$ is a \emph{retraction} of $X$ if there exists a simple node $x\in X$ such that $X'=X \setminus D_x$ with the same order relations that are on $X$.
\end{definition}
\begin{remark}\label{posets are isomorphic}
    Notice that if $X'$ is a retraction of $X$, then for any $y \in X'$, the posets $y_X^\uparrow$ and $y_{X'}^\uparrow$ are naturally isomorphic.
\end{remark}
Suppose that $X'$ is a retraction of $X$. Since the construction of a dimension-preserving saturated embedding will require us to keep track of the dimension every time we perform a gluing or retraction operation on $X$, a natural question arises regarding how the dimensions of $X$ and $X'$ are related. We use the next lemma to show that retractions either preserve dimension or decrease it by exactly one.
\begin{lemma}\label{retraction and node naturally poset isomorphic}
    Let $X, X', x$ and $D_x$ be as in Definition $\ref{retraction definition}$. If $\dim(X)=\dim(u^{\uparrow}_{X})$ for some $u \in D_x$, then $\dim(X)=\dim(x_{X'}^{\uparrow})+1$.
\end{lemma}
\begin{proof}
  By Remark $\ref{posets are isomorphic}$, $x \in X'$ immediately implies $\dim(x_{X'}^{\uparrow}) = \dim(x_X^{\uparrow})$. Let $u \in D_x$ be a minimal element of $X$ for which $\dim(X) = \dim(u^{\uparrow}_{X})$. Then, because $x$ is the \emph{only} element of $X$ that covers $u$, we have $
\dim(X) = \dim(u^{\uparrow}_{X}) = \dim(x_X^{\uparrow}) + 1 = \dim(x_{X'}^{\uparrow}) + 1.$
\end{proof}

We next define a splitting operation on finite posets, which can informally be thought of as the opposite of the gluing operation. We conclude this section by stating a striking theorem about posets proved in \cite{Colbert}: every finite poset with a unique maximal node can be reduced to a single point through a sequence of splittings and retractions.
 
\begin{definition}[Definition 5.11, \cite{Colbert}] If $Y$ is a finite poset, let $n(Y)$ denote the number of height zero nodes of $Y$ that have at least two covers.
\end{definition}
\begin{definition}[Definition 5.12, \cite{Colbert}]
    Let $Y$ be a finite poset. We say $X$ is a \emph{height zero splitting} of $Y$ if the following statements hold:
    \begin{enumerate}[(i)]
        \item $Y$ is a height zero gluing of $X,$ and
        \item either $n(X)=0$ or $n(X)<n(Y)$.
    \end{enumerate}
\end{definition}
\begin{definition}[Definition 5.18, \cite{Colbert}]
    A sequence $(X_1,\dots , X_n)$ of posets is called a \emph{reduction sequence} if for all $1 \leq i < n$, $X_i$ is a height zero splitting of $X_{i+1}$, a retraction of $X_{i+1}$, or equal to $X_{i+1}$. If $X$ and $Y$ are posets, we say $Y$ \emph{reduces to} $X$ if there exists a reduction sequence of the form $(X_1 = X,\dots, X_n = Y )$.
\end{definition}
\begin{theorem}[Theorem 5.19, \cite{Colbert}]\label{every finite poset reduced to a point}
    Every finite poset with a unique maximal node reduces to a point.
\end{theorem}

\section{Gluing and Retraction in Rings}\label{Gluing and Retraction in Rings}

%We begin by discussing the ring-theoretic analogues of gluing and splitting. In the latter half of the section, we present key results from \cite{Colbert} along with remarks on the main ideas behind their proofs.
%\subsubsection{Growing and Gluing Theorems}
In this section, we discuss the ring-theoretic analogues of the gluing operation and the retraction operation.
The majority of this section is dedicated to a detailed exploration of Theorem~\(\ref{gluing theorem}\), \emph{The Gluing Theorem} from \cite{gluingpaper}. Specifically, given a local ring \( S \) satisfying certain conditions, Theorem~\(\ref{gluing theorem}\) shows the existence of a local subring $B$ of $S$ where, as partially ordered sets, the spectrum of $B$ is the same as the spectrum of $S$ except that in $B$ some of the minimal prime ideals of $S$ are ``glued together." In Lemma~\(\ref{ass and subrings, good remain good}\), we study the relationship between the associated prime ideals of \( B/bB \) and those of \( S/bS \) where $b$ is a regular element of $B$, which will be useful when we construct our final local unique factorization domain. In Lemma~\(\ref{lifting chaing from B to S}\), we demonstrate a close connection between the rings $S$ and $B$ by comparing the coheights of their minimal prime ideals. We then state \emph{The Growing Theorem} from \cite{Colbert} (Theorem $\ref{the growing theorem}$), which can be thought of as the ring version of the reverse of the retraction operation. Finally, we present two key results from \cite{Colbert} regarding retractions and height zero gluings, including remarks highlighting the central ideas of their proofs. These lemmas are essential in the proof in \cite{Colbert} that any finite poset is isomorphic to a saturated subset of the prime spectrum of a quasi-excellent domain (Theorem~$\ref{theorem 6.5 in colbert}$). In Section~$\ref{quasi-excellent ring construction}$, we generalize these lemmas to establish a stronger form of this theorem.

\begin{theorem}[Theorem 2.14, \cite{gluingpaper}]\label{gluing theorem}(The Gluing Theorem). Let $(S,M)$ be a reduced local ring containing the rationals with $S/M$ uncountable and $|S| = |S/M|.$ Suppose $\Min(S)$ is partitioned into $m \geq 1$ subcollections $C_1, \dots , C_m$. Then there is a reduced local ring $B \subseteq S$ with maximal ideal $B \cap M$ such that:
    \begin{enumerate}
       \item $B$ contains the rationals,
        \item $\widehat{B} = \widehat{S}$,
        \item $B/(B \cap M )$ is uncountable and $|B| = |B/(B \cap M )|,$
        \item For $Q,Q' \in \Min(S)$, $Q\cap B=Q'\cap B$ if and only if there is an $i\in\{1,\dots,m\}$ with $Q\in C_i$ and $Q' \in C_i$,
        \item\label{property about f map} The map $f : \spec(S) \longrightarrow \spec(B)$ given by $f(P) = B \cap P$ is onto and, if $P$ is a prime ideal of $S$ with positive height, then $f(P)S = P$. In particular, if $P$ and $P'$ are prime ideals of $S$ with positive height, then $f(P)$ has positive height and $f(P) = f(P')$ implies that $P = P'$.
   \end{enumerate}
\end{theorem}
\begin{remark}\label{nicecorrespondence}
    Let $B$ and $S$ be as in Theorem $\ref{gluing theorem}$. Condition $\ref{property about f map}$ establishes that $f$ is an inclusion preserving bijection from the positive height prime ideals of $B$ to the positive height prime ideals of $S$. Thus, there is a one-to-one inclusion preserving correspondence between the prime ideals of $S$ of positive height and the prime ideals of $B$ of positive height.
\end{remark}

Suppose $S$ and $B$ are as in Theorem \ref{gluing theorem}, and let $b$ be a regular element of $B$. In Lemma \ref{ass and subrings, good remain good}, we show that if $P$ is a prime ideal of $S$ satisfying $P \cap B \in \ass_B(B/bB)$, then $P \in \ass_S(S/bS)$. To do this, we first recall a preliminary result.
%The following lemmas describe how the associated primes of two distinct $R$-modules are related. This connection plays a key role in the proof of Lemma~$\ref{ass and subrings, good remain good}$.
%\begin{prop}\label{ass of subrings}
%    Let $R$ be a ring and $M, M'$ $R-$modules with $M'\hookrightarrow M$. Then
%    $$\ass_R(M') \subset \ass_R(M)$$ 
%\end{prop}
%\begin{proof}
%    Let $P \in \ass_R(M')$. Then $P= \text{ann}(m')$ for some $m' \in M'$. Since $M'\hookrightarrow M$, $P$ is the annahilator of the image of $m'$ in $M$. 
%\end{proof}
\begin{lemma}[Exercise 6.7, \cite{Matsumura}]\label{changing base rings for ass}
    Let $\phi:R \longrightarrow S$ be a homomorphism of Noetherian rings. Let $M$ be a finite $S-$module, and $\phi_*: \spec(S) \longrightarrow \spec(R)$ be given by $\phi_*(P)=\phi^{-1}(P)$. Then $\phi_*(\ass_S(M))=\ass_R(M)$.
\end{lemma}
\begin{lemma}\label{ass and subrings, good remain good}
    Let $(S,M)$ and $(B,B \cap M)$ be as in Theorem $\ref{gluing theorem}$ and let $b$ be a regular element of $B$. If $P \in \spec(S)$ with $P \cap B \in \ass_B(B/bB)$, then $P \in \ass_S(S/bS)$.
\end{lemma}
\begin{proof}
    Let $f : \spec(S) \longrightarrow \spec(B)$ be given by $f(P) = B \cap P$. 
%    We claim that $B/bB \hookrightarrow S/bS$. Consider the canonical injection and projection $i: B \longrightarrow S$ and $\pi: S \longrightarrow S/bS$.
%    Let $\varphi:=\pi \circ i: B \longrightarrow S/bS$ and note that $ker(\varphi)=bS\cap B$. 
    Since $\widehat{S}=\widehat{B}=T$, we have $bT \cap S=bS$ and $bT\cap B=bB$. Thus, $bS\cap B=bT \cap S \cap B=bT\cap B=bB$, and so the kernel of the map $B \longrightarrow S/bS$ is $bB$, and it follows that $B/bB \longrightarrow S/bS$ is an injection.
  Therefore,
    $$\ass_B(B/bB) \subseteq \ass_B(S/bS).$$

Now suppose that $P \in \Spec(S)$ with $P \cap B \in \ass_B(B/bB)$. Then $P \cap B \in \ass_B(S/bS)$. By Lemma $\ref{changing base rings for ass}$, $f(\ass_S(S/bS))=\ass_B(S/bS)$, and so 
%In particular, for any $\mathfrak{q} \in \mathrm{Spec}(B)$, 
%$\mathfrak{q} \in \mathrm{Ass}_B(S/bS)$ holds if and only if
there exists $\mathfrak{p} \in \mathrm{Spec}(S)$ with $\mathfrak{p} \in \mathrm{Ass}_S(S/bS)$ such that 
%$\mathfrak{q} = \mathfrak{p} \cap B$ and $\mathfrak{p} \in \mathrm{Ass}_S(S/bS)$. Applying this to $P \cap B \in \mathrm{Ass}_B(S/bS)$ guarantees the existence of $\mathfrak{p} \in \mathrm{Spec}(S)$ with
$\mathfrak{p} \cap B = P \cap B$.  
Since $\widehat{B} = \widehat{S}$, the extension $B \longrightarrow S$ is flat, and so $b \in B$ is a regular element of $S$. Thus, $\mathfrak{p}$ and $P$ are not minimal prime ideals of $S$. Since $f(P) = f(\mathfrak{p}),$ condition \ref{property about f map} of Theorem \ref{gluing theorem} gives us $P = \mathfrak{p}$ and so $P \in \mathrm{Ass}_S(S/bS)$.
%Finally, since Theorem $\ref{gluing theorem}$ establishes an inclusion-preserving bijection between $\mathrm{Spec}(B) \setminus \{(0)\}$ and $\mathrm{Spec}(S) \setminus \mathrm{Min}(S)$, $P \cap B= \mathfrak{p} \cap B$ implies that $P = \mathfrak{p}$. Thus, $P \in \mathrm{Ass}_S(S/bS)$.
\end{proof}
Lemma~$\ref{ass and subrings, good remain good}$ holds because of the unique relationship between the rings $S$ and $B$. Generally speaking, the lemma does not hold for a ring $R$ and a subring $R'$ of $R$. To illustrate this point, consider the following simple example: set $R=\mathbb{Q}[[x,y]]$, $R'=\mathbb{Q}[[x]]$, $b=x$, and $P=(y)$. Notice that $P \cap R' = (0) \in \operatorname{Ass}_{R'}(R'/xR')$ and $P \notin \operatorname{Ass}_R(R/xR)$.

We now work to establish a relationship between the coheights of the minimal prime ideals of $S$ and the coheights of the minimal prime ideals of $B$ where $S$ and $B$ are as in Theorem \ref{gluing theorem}. We first state a useful result from \cite{Colbert}. Note that the hypotheses of Lemma \ref{lemma about existance of minimal prime in S} hold for the rings $S$ and $B$ from Theorem \ref{gluing theorem}.

\begin{lemma}[Lemma 6.3, \cite{Colbert}]\label{lemma about existance of minimal prime in S}
    Let $(S,M)$ and $(B,B\cap M)$ be two reduced local rings such that $B\subseteq S$ and $\widehat{B}=\widehat{S}.$ Suppose that $f : \spec(S) \longrightarrow \spec(B)$ given by $f(P) := P \cap B$ is surjective and satisfies $f(P)S = P$ for all prime ideals $P$ of $S$ whose height is positive. Then we have the following:
    \begin{enumerate}
        \item If $Q$ is a minimal prime ideal of $S$, then $Q\cap B$ is a minimal prime ideal of $B$.
        \item\label{condition about minimals in lemma 6.3}If $P_1 \cap B = P_2 \cap B$ for prime ideals $P_1, P_2$ of $S$, then either $P_1 = P_2$ or both $P_1$ and $P_2$ are minimal prime ideals of $S$.
        \item If $P_1 \cap B \supseteq P_2 \cap B$ for prime ideals $P_1, P_2$ of $S$, then there exists a prime ideal $Q$ of $S$ such that $P_1 \supseteq Q$ and $Q \cap B = P_2 \cap B$.
    \end{enumerate}
\end{lemma}

\begin{lemma}\label{lifting chaing from B to S}
    Let $(S,M)$ and $(B,B \cap M)$ be as in Theorem $\ref{gluing theorem}$. Suppose that $\mathfrak{p} \in \Min(B)$. Then, there exists $P \in \Min(S)$ with $P \cap B=\mathfrak{p}$ such that $\dim(S/P)=\dim(B/\mathfrak{p})$.
\end{lemma}
\begin{proof}
Let $\dim(B/\mathfrak{p}) = m$, and suppose that $\mathfrak{p} = \mathfrak{p}_0 \subsetneq \mathfrak{p}_1 \subsetneq \dots \subsetneq \mathfrak{p}_m = B \cap M$ is a saturated chain of prime ideals of $B$. Since the map $f: \spec(S) \to \spec(B)$ defined by $f(P) = P \cap B$ is surjective, for each $i = 0, \dots, m$ there exists a prime ideal $P_i \in \spec(S)$ with $\mathfrak{p}_i = P_i \cap B$.

Because there is an inclusion-preserving bijection between the prime ideals of positive height in $B$ and in $S$, we have $P_1\subsetneq P_2\subsetneq\dots\subsetneq P_m=M$ where $P_i$ has positive height for all $i = 1,2, \ldots ,m$. Moreover, by Lemma \ref{lemma about existance of minimal prime in S}, there exists a prime ideal $Q$ of $S$ such that $Q\cap B=\mathfrak{p}$ and $Q\subseteq P_1$. Note that since $\mathfrak{p}$ is minimal and $Q \cap B = \mathfrak{p}$, condition \ref{property about f map} of Theorem \ref{gluing theorem} implies that $Q$ is a minimal prime ideal of $S$. If $Q = P_1$, then $P_1$ is a minimal prime ideal of $S$, a contradiction.
%We now claim that the inclusion $Q\subseteq P_1$ is saturated.
%Assume for contradiction that there exists a prime ideal $Q_1$ of $S$ with $Q\subsetneq Q_1\subsetneq P_1$. Since the saturated chain $\mathfrak{p}=Q\cap B\subsetneq P_1\cap B$ in $B$ admits no intermediate prime ideal, it must be that either $Q\cap B=Q_1\cap B$ or $Q_1\cap B=P_1\cap B$. Noting that $Q_1$ is not a minimal prime ideal of $S$, both cases violate condition \ref{condition about minimals in lemma 6.3} of Lemma \ref{lemma about existance of minimal prime in S}.
Thus, $Q \subsetneq P_1 \subsetneq P_2 \subsetneq \dots \subsetneq P_m = M$, and it follows that $\dim(B/\mathfrak{p}) \leq \dim(S/Q)$. 

Now, suppose that $\dim(S/Q) = n$ and consider a saturated chain of prime ideals $Q \subsetneq Q_1 \subsetneq Q_2 \subsetneq \dots \subsetneq Q_n = M$ in $S$. Then the inclusion-preserving bijection between the prime ideals of positive height in $B$ and $S$ results in the chain of prime ideals $\mathfrak{p} = Q \cap B \subseteq Q_1 \cap B \subsetneq Q_2 \cap B \subsetneq \dots \subsetneq Q_n \cap B = M \cap B$ of $B$. If $Q \cap B = Q_1 \cap B$, then condition \ref{condition about minimals in lemma 6.3} of Lemma \ref{lemma about existance of minimal prime in S} implies that $Q_1$ is a minimal prime ideal of $S$, a contradiction. Thus, we have $\mathfrak{p} = Q \cap B \subsetneq Q_1 \cap B \subsetneq Q_2 \cap B \subsetneq \dots \subsetneq Q_n \cap B = M \cap B$ showing that $\dim(S/Q) \leq \dim(B/\mathfrak{p})$. Therefore, letting $P=Q$, we obtain a minimal prime ideal $P$ of $S$ with $\dim(S/P)=\dim(B/\mathfrak{p})$ and $P \cap B=\mathfrak{p}$.
\end{proof}

We now state The Growing Theorem from \cite{Colbert}, which is the ring-theoretic version of the opposite of the retraction operation.
\begin{theorem}[Theorem 3.1, \cite{Colbert}]\label{the growing theorem}(The Growing Theorem). Let $B$ be a local ring containing an infinite field $K.$ Let $\Min(B) = \{P_1, P_2, \dots, P_m\}$. Let $n$ be a positive integer, let $y$ and $z$ be indeterminates, and let $A = B[[y, z]]$. Then there are distinct prime ideals $Q_1, Q_2,\dots, Q_n$ of $A$ such that
   \begin{enumerate}
       \item For every $i = 1,2,\dots,n$, we have $Q_i \subseteq (P_j,y,z)A$ if and only if $j = 1$.
       \item For every $i = 1,2,\dots,n$ the prime ideal $(P_1,y,z)A/Q_i$ in the ring $A/Q_i$ has height one, and
       \item If $J = (\bigcap_{i=1}^nQ_i) \cap P_2' \cap \cdots \cap P_m'$ where $P_j' = (P_j,y,z)A$ then the minimal prime ideals of $S = A/J$ are $\{Q_1/J,Q_2/J,\dots,Q_n/J,P_2'/J,\dots,P_m' /J\}$, and $(P_1,y,z)/J$ is a height one prime ideal of $S$ containing exactly $n$ minimal prime ideals of $S$, namely $\{Q_1/J,Q_2/J,\dots,Q_n/J\}$.
   \end{enumerate}
\end{theorem}
%\subsubsection{Retraction and Splitting Theorems}
%In this section, we present two key results from \cite{Colbert} regarding retractions and height-zero splittings, including remarks highlighting the central ideas of their proofs. These lemmas are essential in the proof by Colbert and Loepp that any finite poset is isomorphic to a saturated subset of the prime spectrum of a quasi-excellent domain (Theorem~$\ref{theorem 6.5 in colbert}$). In Section~$\ref{quasi-excellent ring construction}$, we will generalize these lemmas to establish a stronger form of this theorem.

We next state two important results from \cite{Colbert}, along with specific commentary about the parts of their proofs that will be useful later. In particular, Theorem \ref{retraction and ring} shows that if $X$ and $Z$ are finite posets, $X$ is a retraction of $Z$, $B$ is a ring satisfying some nice properties and there is a saturated embedding from $X$ to $\Spec(B)$, then one can find a ring $S$ satisfying some nice properties such that there is a saturated embedding from $Z$ to $\Spec(S)$.  Theorem \ref{height zero gluing theorem} is the analogous result for height zero gluings.

\begin{theorem}[Theorem 6.2, \cite{Colbert}]\label{retraction and ring}
    Let $X$ and $Z$ be finite posets, and suppose that $X$ is a retraction of $Z.$ Let $B$ be a reduced local ring containing the rationals with $|B| = |B/M| = c$. Suppose $\psi : X \longrightarrow \spec(B)$ is a saturated embedding such that $\psi(\min X) = \Min(B)$. Then there exists a local ring S such that
     \begin{enumerate}
     \item  $S$ is reduced,
     \item $S$ contains the rationals,
     \item If $N$ is the maximal ideal of $S$, then $|S| = |S/N| = c$,
     \item There is a saturated embedding $\phi : Z \longrightarrow \spec (S)$ such that $\phi(\min Z ) = \Min(S)$.   
 \end{enumerate}
Moreover, if $B$ is quasi-excellent, then so is $S$.
\end{theorem}
\begin{remark}\label{remark about retraction and poset}
Let $X, Z, B$, and $S$ be as in Theorem \ref{retraction and ring}. Then, since $X$ is a retraction of $Z$, we have $X = Z \setminus D_p$, where $p$ is a simple node in $Z$ and $D_p =\{u \in Z : u <_c p\}$. Let $\min X = \{p, x_2, \dots, x_n\}$ and let $D_p = \{q_1, \dots, q_m\}$. Since $\psi(\min X) = \Min(B)$, we may write $\Min(B) = \{\psi(p) := P_1, \psi(x_2) := P_2, \dots, \psi(x_n) := P_n\}$.

The ring $S$ in Theorem \ref{retraction and ring} is constructed by applying Theorem \ref{the growing theorem} to $B$. Explicitly, we have $S = B[[y,z]]/J$ for the ideal $J$ described in Theorem \ref{the growing theorem}. The minimal prime ideals of $S$ are given by $$\Min(S) = \{Q_1/J, \dots, Q_m/J, P_2'/J, \dots, P_n'/J\},$$ where for each $i = 2, \dots, n$ and $j = 1, \dots, m$, $P_i' = (P_i, y, z)$ and $Q_j=(P_1,y+\beta_jz)$ where $\beta_1, \beta_2, \ldots ,\beta_m$ are distinct nonzero rationals. 
% . Furthermore, for each $j = 1, \dots, m$ and each $i = 2, \dots, n$, the inclusion $Q_j/J \subseteq P_i'/J$ holds if and only if $i = 1$, and $P_1'/J$ is a height-one prime ideal of $S$. 
Moreover, the saturated embedding $\phi: Z \longrightarrow \spec(S)$ is explicitly defined by:
   \[
\phi(x) = \begin{cases} 
       (\psi(x), y, z)/J & x \in X, \\
       Q_i/J & x=q_i \in D_p.
   \end{cases}
\]

\end{remark}

\begin{theorem}[Theorem 6.4, \cite{Colbert}]\label{height zero gluing theorem}
    Let $Z$ be a nonempty finite poset, and let $Y$ be a height zero gluing of $Z$. Let $(S, M)$ be a reduced local ring containing $\Q$ such that $S/M$ is uncountable and $|S| = |S/M|$. Suppose $\psi : Z \longrightarrow \spec(S)$ is a saturated embedding such that $\psi(\min Z) = \Min(S)$. Then there exists a reduced local ring $B \subseteq S$ with maximal ideal $B \cap M$ such that:
         \begin{enumerate}
     \item $B$ contains $\Q$,
     \item $\widehat{B}=\widehat{S}$,
     \item $B/(B \cap M )$ is uncountable, $|B| = |B/(B \cap M )|$, and
     \item There is a saturated embedding $\varphi : Y \longrightarrow \spec (B)$ such that $\varphi(\min Y ) = \Min(B)$.   
 \end{enumerate}
Moreover, if $S$ is quasi-excellent, then so is $B$.
\end{theorem}
\begin{remark}\label{remark about properties of height zero gluing}
The ring $B$ in Theorem \ref{height zero gluing theorem} is constructed by applying Theorem \ref{gluing theorem} to a partition of $\Min(S)$ induced by the map $\psi$. More precisely, let $g$ denote the gluing map from $Z$ to $Y$ along $C \subseteq \min Z$. If $C$ is empty, then $B = S$. If $C = \min Z$, we define the partition as $\mathcal{P} = \{\psi(C)\} = \{\psi(\min Z)\}$. Otherwise, we set $\mathcal{P} = \{\psi(C)\} \cup \{\{\psi(x)\} \mid x \in \min Z \setminus C\}$. By applying Theorem \ref{gluing theorem} to this partition, we ensure that the map $ f : \spec(S) \to \spec(B)$ given by $f(P) = P \cap B, $ is surjective and satisfies $f(P)S = P $ for every prime ideal $P$ of $S$ of positive height. Finally, it is shown that the composition $f \circ \psi$ is compatible with $g$. Consequently, $\varphi$ is the unique poset map that satisfies the identity $ f \circ \psi = \varphi \circ g. $

\end{remark}

In \cite{Colbert}, the authors use Theorem \ref{every finite poset reduced to a point}, Theorem \ref{retraction and ring}, and Theorem \ref{height zero gluing theorem} to prove the following result. The goal of Section \ref{quasi-excellent ring construction} in this paper is to prove a stronger form of Theorem \ref{theorem 6.5 in colbert} (see Theorem \ref{local quasi-excellent ring codimension preserving theorem}).
 
\begin{theorem}[Theorem 6.5, \cite{Colbert}]\label{theorem 6.5 in colbert}
     If $X$ is a finite poset, then there exists a local domain $(B,M)$ and a saturated embedding $\Psi : X \longrightarrow \spec(B)$ such that:
     \begin{enumerate}[(i)]
     \item $B$ contains $\Q,$
     \item $|B| = |B/(B \cap M )|=c$, and
     \item $B$ is quasi-excellent.
     \end{enumerate}
\end{theorem}

\section{Local Quasi-Excellent Ring Construction}\label{quasi-excellent ring construction}

In this section, we generalize Theorem \ref{theorem 6.5 in colbert} in the following way. Let $Y$ be a finite poset with exactly one maximal element. In Theorem \ref{local quasi-excellent ring codimension preserving theorem}, we show that there is a quasi-excellent local ring $S$ and a saturated embedding $\Psi : Y \longrightarrow \Spec(S)$ such that $\Psi(\min Y) = \Min(S)$.  Moreover, if $P$ is a minimal prime ideal of $S$, then there is a minimal element $z$ of $Y$ such that $\Psi(z) = P$ and the coheight of $\Psi(z)$ is the same as the dimension of $z^{\uparrow}_Y$. In other words, there is a one-to-one correspondence between the minimal elements of $Y$ and the minimal prime ideals of $S$ and, for every minimal element $z$ of $Y$, the longest chain from $z$ to the maximal element of $Y$ is the same as the longest chain of prime ideals from $\Psi(z)$ to the maximal ideal of $S$.
 Theorem $\ref{local quasi-excellent ring codimension preserving theorem}$ establishes an important step in proving Theorem~$\ref{main theorem of characterization of embeddings of UFDs}$, the main result of this paper.

Our proof strategy for Theorem \ref{local quasi-excellent ring codimension preserving theorem} is to adjust the construction of the quasi-excellent ring $B$ and the saturated embedding $\Psi: X \longrightarrow \spec(B)$ in Theorem~$\ref{theorem 6.5 in colbert}$ so that $\Psi$ satisfies our additional desired properties. We begin with an important definition.

\begin{definition}
    Let $R$ be a Noetherian ring, and let $D \subseteq \spec(R)$. Let $X$ be a poset and let $\varphi: X \longrightarrow \spec(R)$ be a saturated embedding. We say $\varphi$ is a \emph{coheight-preserving saturated embedding} along $D$ if for every $P \in D$, there is an element $x \in X$ with $\varphi(x)=P$ and $\dim(R/P)=\dim(x_X^{\uparrow})$. 
\end{definition}

\begin{lemma}\label{MinsgotoMins}
Let $R$ be a Noetherian ring, $X$ a finite poset and $D = \Min(R)$.  Suppose $\phi : X \longrightarrow \Spec(R)$ is a coheight-preserving saturated embedding along $D$.  Then $\phi(\min X) = \Min(R)$.
\end{lemma}

\begin{proof}
Suppose $P \in \Min(R)$. Then there exists $x \in X$ with $\phi(x) = P$.  Let $x_0$ be an element of $\min X$ with $x_0 \leq x$. As $\phi$ is a poset map, $\phi(x_0) \subseteq \phi(x) = P$. Since $P$ is a minimal prime ideal of $R$, we have $\phi(x_0) = \phi(x)$ and so since $\phi$ is a poset embedding, $x_0 = x$. It follows that $x \in \min X$ and we have $\Min(R) \subseteq \phi(\min X)$. 

Suppose $y \in \min X$. Then there is a minimal prime ideal $P$ of $R$ such that $P \subseteq \phi(y)$. Now, there exists $x \in X$ with $\phi(x) = P$. Since $\phi$ is a poset embedding and $\phi(x) \subseteq \phi(y)$, we have $x \leq y$. As $y$ is a minimal element of $X$, $x = y$ and so $\phi(y) = \phi(x) = P$. Thus $\phi(\min X) \subseteq \Min(R)$ and the result follows.
\end{proof}

\begin{remark}\label{remark about dimension and codimension preserving}
Let $R$ be a Noetherian ring, $X$ a finite poset and $D = \Min(R)$.  Suppose $\phi : X \longrightarrow \Spec(R)$ is a coheight-preserving saturated embedding along $D$. By Lemma \ref{MinsgotoMins}, $\phi(\min X) = \Min(R)$ and so we can write $\min X = \{x_1,\dots,x_n\}$ and correspondingly $\Min(R) = \{P_1 := \psi(x_1), \dots, P_n := \psi(x_n)\}$. Now,
$$
\dim(R) = \max\{\dim(R/P_1), \dots, \dim(R/P_n)\}.
$$
By the coheight-preserving property, we obtain
$$
\dim(R) = \max\{\dim(x{_1}^{\uparrow}_X), \dots, \dim(x{_n}^{\uparrow}_X)\} = \dim(X),
$$ and it follows that $\phi$ is a dimension-preserving saturated embedding.
\end{remark}

%The following remark illustrates why a codimension-preserving saturated embedding along minimal primes is a stronger condition than a dimension-preserving saturated embedding.

%\begin{remark}\label{remark about dimension and codimension preserving} Suppose $R$ is a Noetherian ring, $X$ is a finite poset, and $\psi: X \longrightarrow \spec(R)$ is a codimension-preserving embedding along $\Min(X)$. Note that if $\psi(\Min(X)) = \Min(R)$, we necessarily have $\dim(X) = \dim(R)$. Indeed, since $\psi(\Min(X)) = \Min(R)$, we can write $\Min(X) = \{x_1,\dots,x_n\}$ and correspondingly $\Min(R) = \{P_1 := \psi(x_1), \dots, P_n := \psi(x_n)\}$. Thus,
%$$
%\dim(R) = \max\{\dim(R/P_1), \dots, \dim(R/P_n)\}.
%$$
%By the codimension-preserving property, we obtain
%$$
%\dim(R) = \max\{\dim(x{_1}^{\uparrow}), \dots, \dim(x{_n}^{\uparrow})\} = \dim(X).
%$$  
%\end{remark}

  We now present two lemmas that are crucial in the proof of Theorem $\ref{local quasi-excellent ring codimension preserving theorem}$.  The first lemma is a generalization of Theorem $\ref{height zero gluing theorem}$ and the second is a generalization of Theorem $\ref{retraction and ring}$. 
 
\begin{lemma}\label{codimension preserving gluing thm}
    Let $Z$ be a finite poset with a unique maximal element, and let $Y$ be a height zero gluing of $Z$. Let $(S, M)$ be a reduced local ring containing $\Q$ such that $S/M$ is uncountable and $|S| = |S/M|$. Suppose $\psi : Z \longrightarrow \spec(S)$ is a coheight-preserving saturated embedding along $\Min(S)$. Then there exists a reduced local ring $B \subseteq S$ with maximal ideal $B \cap M$ such that:
         \begin{enumerate}
     \item\label{condition 1 codimension preserving gluing thm} $B$ contains $\Q$,
     \item $\widehat{B}=\widehat{S}$,
     \item \label{condition 3 codimension preserving gluing thm}$B/(B \cap M )$ is uncountable, $|B| = |B/(B \cap M )|$, and
     \item there is a coheight-preserving saturated embedding $\phi : Y \longrightarrow \spec (B)$ along $\Min(B)$.
 \end{enumerate}
Moreover, if $S$ is quasi-excellent, then so is $B$.
\end{lemma}
\begin{proof}
 By Theorem \ref{height zero gluing theorem} and Remark $\ref{remark about properties of height zero gluing}$, there exists a reduced local ring $B \subseteq S$ obtained from Theorem \ref{gluing theorem} that satisfies conditions \ref{condition 1 codimension preserving gluing thm}--\ref{condition 3 codimension preserving gluing thm} and satisfies the property that if $S$ is quasi-excellent, then so is $B$. Moreover, there is a saturated embedding $\phi : Y \longrightarrow \spec(B)$ such that the following diagram commutes:
 \begin{equation}\label{commutative diagram height zero gluing}
\begin{tikzcd}
	{{\spec(S)}} && {{\spec(B)}} \\
	\\
	Z && Y
	\arrow["f", from=1-1, to=1-3]
	\arrow["\psi", from=3-1, to=1-1]
	\arrow["g"', from=3-1, to=3-3]
	\arrow["\phi"', from=3-3, to=1-3]
\end{tikzcd}
 \end{equation}
Here, $g$ is the gluing map from $Z$ to $Y$, and the map $f$, defined by $f(P)=P \cap B$, is surjective. It remains to show that $\phi$ is coheight-preserving along $\Min (B)$.

 Fix $\mathfrak{p}_0 \in \Min(B)$ and suppose that $\dim(B/\mathfrak{p}_0)=n$. By Lemma $\ref{lifting chaing from B to S}$, there exists $P_0 \in \Min(S)$ with $\dim(S/P_0) = n$ and $P_0 \cap B=\mathfrak{p}_0$. Since $\psi$ is coheight-preserving along $\Min(S)$, there exists $z_0 \in Z$ with $\psi(z_0)=P_0$ and $\dim(z{_0}_Z^{\uparrow}) = \dim(S/P_0) = n$. By Lemma \ref{MinsgotoMins}, $\psi(\min Z) = \Min(S)$. Since $\psi$ is injective, $z_0$ is a minimal element of $Z$. Define $y_0 = g(z_0)$. We claim that $\phi(y_0) = \mathfrak{p}_0$ and $\dim(y{_0}_Y^{\uparrow}) = \dim(B/\mathfrak{p}_0)$. Suppose $\dim(y{_0}_Y^{\uparrow}) = m$. Because $g$ is surjective, there is a saturated chain in $Y$ of length $m$ given by
 \begin{equation}\label{saturated chain in Y}
   g(z_0)<_c g(x_1)<_c\dots<_cg(z_m)=g(z)
 \end{equation}
 where $z$ is the unique maximal node of $Z$.  By commutativity of diagram (\ref{commutative diagram height zero gluing}), we have $\phi(g(z_0))=f(\psi(z_0))=f(P_0)=P_0 \cap B=\mathfrak{p}_0$, and so $\phi(y_0) = \mathfrak{p}_0$. Since $\phi$ is a saturated embedding, the image of chain ($\ref{saturated chain in Y}$) in $\spec(B)$ is a saturated chain of length $m$ starting at $\mathfrak{p}_0$. Thus, $\dim(B/\mathfrak{p}_0) = n \geq m$. On the other hand, by Lemma $\ref{dimension of gluing poset}$, $\dim(g(z{_0})_Y^{\uparrow})=\text{max}\{\dim(z_Z^{\uparrow})\mid g(z)=g(z_0)\} \geq \dim(z{_0}_Z^{\uparrow})=n$. Thus, $m \geq n$, implying that $m = n$. Hence, $\phi(y_0)=\mathfrak{p}_0$ and $\dim(B/\mathfrak{p}_0) = \dim(y{_0}_Y^{\uparrow})$ as claimed. It follows that $\phi$ is coheight-preserving along $\Min (B)$.
 \end{proof}

\begin{lemma}\label{retraction lemma for quasi-excellent ring}
    Let $X$ and $Z$ be finite posets, and suppose $X$ is a retraction of $Z.$ Let $B$ be a reduced local ring containing the rationals with $|B| = |B/M| = c$. Suppose $\psi : X \longrightarrow \spec(B)$ is a coheight-preserving saturated embedding along $\Min(B)$. Then there exists a local ring S such that
     \begin{enumerate}
     \item\label{condition 1 in retraction codimension preserving}  $S$ is reduced,
     \item $S$ contains the rationals,
     \item\label{condition 3 in retraction codimension preserving} if $N$ is the maximal ideal of $S$, then $|S| = |S/N| = c$, and
     \item there is a coheight-preserving saturated embedding $\phi : Z \longrightarrow \spec (S)$ along $\Min(S)$.   
 \end{enumerate}
Moreover, if $B$ is quasi-excellent, then so is $S$.
\end{lemma}
\begin{proof}
   By Theorem $\ref{retraction and ring}$, there exists a local ring $S$ such that conditions $\ref{condition 1 in retraction codimension preserving}-\ref{condition 3 in retraction codimension preserving}$ are satisfied and a saturated embedding $\phi : Z \longrightarrow \Spec(S)$ such that $\phi(\min Z ) = \Min(S)$. Moreover, if $B$ is quasi-excellent, then so is $S$. It remains to show that $\phi$ is coheight-preserving along $\Min(S)$.

   Since $X$ is a retraction of $Z$, by Definition $\ref{retraction definition}$, $X = Z \setminus D_p$, where $p$ is a simple node in $Z$. If we let $\min X = \{p, x_2, \dots, x_n\}$ and $D_p = \{q_1, \dots, q_m\}$, we can write $\Min(B) = \{\psi(p) := P_1, \psi(x_2) := P_2, \dots, \psi(x_n) := P_n\}$ as $\psi(\min X) = \Min(B)$ by Lemma \ref{MinsgotoMins}. By Remark $\ref{remark about retraction and poset}$, $S = B[[t_1,t_2]]/J$ for some ideal $J$ and $\Min(S) = \{Q_1/J, \dots, Q_m/J, P_2'/J, \dots, P_n'/J\}$, where $P_i' = (P_i, t_1, t_2)$ for each $i = 2, \dots, n$ and $Q_j=(P_1,t_1+\beta_j t_2)$ for each $j=1,\dots, m$ where $\beta_1, \ldots ,\beta_m$ are distinct nonzero rationals. Moreover, the saturated embedding $\phi: Z \longrightarrow \spec(S)$ is explicitly defined by:
   \[
\phi(x) = \begin{cases} 
       (\psi(x), t_1, t_2)/J & x \in X, \\
       Q_i/J & x=q_i \in D_p.
   \end{cases}
\]
Let $P \in \Min(S)$. Our goal is to find an element $z \in Z$ such that $\phi(z) = P$ and $\dim(S/P) = \dim(z_Z^{\uparrow})$. First, assume $P = P_i'/J$ for some $i = 2, \dots, n$. In this case, we have 
$$S/P = (B[[t_1, t_2]]/J)/((P_i, t_1, t_2)/J) \cong B/P_i,$$ 
and so $\dim(S/P) = \dim(B/P_i).$ By assumption, the map $\psi$ is coheight-preserving along $\Min(B)$. So, there exists $x \in X$ with $\psi(x) = P_i$ and $\dim(B/P_i) = \dim(x_X^{\uparrow})$. Because $\psi$ is a saturated embedding, it is injective and so $\psi(x) = P_i = \psi(x_i)$ implies $x = x_i$. Since $x_i \in X$, Remark \ref{posets are isomorphic} ensures that $x{_i}_X^{\uparrow}$ is poset-isomorphic to $x{_i}_Z^{\uparrow}$. Setting $z = x_i$, we obtain $\phi(z) = \phi(x_i) = (\psi(x_i), t_1, t_2)/J = P_i'/J = P$ and $$\dim(S/P) = \dim(B/P_i) = \dim(x{_i}_X^{\uparrow}) = \dim(x{_i}_Z^{\uparrow})=\dim(z_Z^{\uparrow}).$$

Now suppose $P = Q_j/J$ for some $j = 1, \dots, m$. We proceed similarly. Since 
$$S/P=(B[[t_1, t_2]]/J)/((P_1,t_1+\beta_j t_2)/J) \cong B/P_1[[t_1]]$$
we see that $\dim(S/P)= \dim(B/P_1) + 1.$ The coheight-preserving property of $\psi$ and its injectivity on minimal nodes of $X$ imply that $\dim(B/P_1) = \dim(p_X^{\uparrow})$. By choosing $z = q_j$, we see that $\phi(z) = \phi(q_j) = Q_j/J = P$. The key observation is that $p_X^{\uparrow}$ is a retraction of $q{_j}_Z^{\uparrow}$ as $p_X^{\uparrow}=q{_j}_Z^{\uparrow} \setminus \{q_j\}$. Thus, we can use Lemma $\ref{retraction and node naturally poset isomorphic}$ to deduce that $\dim(q{_j}_Z^{\uparrow})=\dim(p_X^{\uparrow})+1$. Therefore,
\[
\dim(S/P) = \dim(B/P_1) + 1 = \dim(p_X^{\uparrow}) + 1 = \dim(q{_j}_Z^{\uparrow})=\dim(z_Z^{\uparrow}).
\] It follows that $\phi$ is coheight-preserving along $\Min(S)$.
\end{proof}

We now have the tools needed to prove the main result of this section.
\begin{theorem}\label{local quasi-excellent ring codimension preserving theorem}
       Let $Y$ be a finite poset with a unique maximal node. Then, there exists a reduced local ring $(S,M)$ and a saturated embedding $\Psi : Y \longrightarrow \spec(S)$ such that:
      \begin{enumerate}
        \item\label{condition 1 B contain sQ} $S$ contains $\Q,$
        \item $|S| = |S/M| = c$,
        \item $S$ is quasi-excellent, and
%        \item $\Psi(\min Y)=\Min(S)$ and,
        \item\label{condition 5 phi is condimension preserving} $\Psi$ is coheight-preserving along $\Min(S)$. 
       \end{enumerate}
\end{theorem}
\begin{proof}
    The proof closely follows the proof of Theorem $\ref{theorem 6.5 in colbert}$. By Theorem $\ref{every finite poset reduced to a point}$, the poset $Y$ reduces to a point. Hence, there exists a reduction sequence $(Y_1,\dots, Y_n = Y)$, where $Y_1$ is a point and each step is either a height zero splitting, a retraction, or simply the identity.

Let $S_1 = \C$, and define $\Psi_1: Y_1 \to \spec(S_1)$ as the poset map that sends the unique element of $Y_1$ to the zero ideal of $S_1$. For $1 \leq i < n$, assume that there is a reduced local ring $(S_i, M_i)$ and a saturated embedding $\Psi_i : Y_i \longrightarrow \Spec(S_i)$  satisfying conditions $\ref{condition 1 B contain sQ}$--$\ref{condition 5 phi is condimension preserving}$ (note that these properties hold for $S_1$ and $\Psi_1$).

We now construct a reduced local ring $(S_{i+1}, M_{i+1})$ and a saturated embedding $\Psi_{i + 1} : Y_{i + 1} \longrightarrow \Spec(S_{i + 1})$ with the same properties. If $Y_{i+1} = Y_i$, set $S_{i+1} = S_i$ and $\Psi_{i+1} = \Psi_i$. If $Y_i$ is a height zero splitting of $Y_{i+1}$—that is, if $Y_{i+1}$ is a height zero gluing of $Y_i$—apply Lemma $\ref{codimension preserving gluing thm}$ with $Y = Y_{i+1}$, $Z = Y_i$, $S = S_i$, and $\psi = \Psi_i$ to obtain $(S_{i+1}, M_{i+1})$ and a saturated embedding $\Psi_{i+1}$ that satisfy the four conditions in Lemma \ref{codimension preserving gluing thm}. Now, $|S_{i + 1}| = |S_{i + 1}/M_{i + 1}| = |\widehat{S_{i + 1}}/M_{i + 1}\widehat{S_{i+ 1}}| = |\widehat{S_{i }}/M_{i }\widehat{S_{i}}| = |S_i/M_i| = c$, and so conditions $\ref{condition 1 B contain sQ}$--$\ref{condition 5 phi is condimension preserving}$ hold.
In the case where $Y_i$ is a retraction of $Y_{i+1}$, use Lemma $\ref{retraction lemma for quasi-excellent ring}$ with $X = Y_i$, $Z = Y_{i+1}$, $B = S_i$, and $\psi = \Psi_i$. This yields a reduced local ring $(S_{i+1}, M_{i+1})$ and a saturated embedding $\Psi_{i+1}$ satisfying conditions $\ref{condition 1 B contain sQ}$--$\ref{condition 5 phi is condimension preserving}$.

Finally, by setting $S := S_n$ and $\Psi := \Psi_n$, we conclude that $\Psi: Y \to \spec(S)$ is a saturated embedding satisfying conditions $\ref{condition 1 B contain sQ}$--$\ref{condition 5 phi is condimension preserving}$.\end{proof}

We end this section with two consequences of Theorem \ref{local quasi-excellent ring codimension preserving theorem}.
\begin{corollary}\label{codimension implies dimension}
    Let $Y$ and $S$ be as in Theorem $\ref{local quasi-excellent ring codimension preserving theorem}$. Then, $\dim(Y)=\dim(S)$.
\end{corollary}
\begin{proof}
    The result follows from Remark $\ref{remark about dimension and codimension preserving}$.\end{proof}

\begin{corollary}\label{corollary quasi-excellent 2}
    Let $Y$ be a finite poset. Then, there exists a reduced Noetherian ring $S$ and a saturated embedding $\Psi : Y \longrightarrow \spec(S)$ such that:
      \begin{enumerate}
        \item $S$ contains $\Q,$
        \item $S$ is quasi-excellent, and
%        \item $\Psi(\min Y)=\Min(S)$ and,
        \item $\Psi$ is coheight-preserving along $\Min(S)$. 
       \end{enumerate}
\end{corollary}
\begin{proof}
We define a new poset $Y' = Y \cup \{z\}$ where the order is given by $y <_{Y'} z$ for $y \in Y$ and $x \leq_{Y'} y$ if and only if $x \leq_Y y$ for $x,y \in Y$. 
%Adjoin to \(Y\) a new greatest element \(\mathbf y\) and equip \(Y\cup\{\mathbf y\}\) with the order extending the original one on \(Y\) and declaring every element of \(Y\) to lie below \(\mathbf y\).  
Then $Y'$ has a unique maximal node, and so by Theorem \ref{local quasi-excellent ring codimension preserving theorem} there exists a reduced quasi‑excellent local ring \((S',M')\) containing the rationals and a coheight‑preserving saturated embedding $\Psi': Y' \longrightarrow \Spec(S')$ along $\Min(S')$.
%along $\Min(Y\cup\{\mathbf y\})$
%\[
%  \Psi_{\mathbf y}\colon Y\cup\{\mathbf y\}\longrightarrow \spec(S_{\mathbf y}),
%\]
%carrying all minimal elements of \(Y\cup\{\mathbf y\}\) bijectively onto the minimal primes of \(S_{\mathbf y}\).  
Let \(y_1,\dots,y_n\) be the maximal elements of \(Y\). 
%then each prime ideal \(\Psi'(y_i)\) is strictly contained in the unique maximal ideal \(M'=\Psi'(z)\). 
If $y \in Y'$ and $y <_c z$, then $y = y_i$ for some $i = 1,2, \ldots ,n$.  Since $\Psi'$ is a saturated embedding, $\Psi'(y_i) \subsetneq \Psi'(z) = M'$ is saturated for every $i = 1,2, \ldots ,n$.

Set \(\mathfrak S = S'\setminus\bigl(\Psi'(y_1)\cup\cdots\cup\Psi'(y_n)\bigr)\) and let \(S=\mathfrak S^{-1}S'\).  Then $S$ is reduced, quasi-excellent, and contains $\mathbb{Q}$.
%Localization preserves reducedness, quasi‑excellence and containment of \(\Q\).
Moreover, there is an inclusion preserving one-to-one correspondence between prime ideals of \(S\) and prime ideals of \(S'\) contained in \(\Psi'(y_i)\) for some $i = 1,2, \ldots ,n$. Thus, \(\Psi\colon Y\to \spec(S)\) defined by $\Psi(y)=\Psi'(y)S$ is a
saturated embedding. Further, since all minimal prime ideals of $S'$ are contained in $\Psi'(y_i)$ for some $i$, $S$ and $S'$ have the same number of minimal prime ideals. 
%It follows that \(\Psi(\min Y)=\Min(S)\).

Let $P \in \Min(S)$, and let $P'$ be the corresponding minimal prime ideal of $S'$. Since $\Psi'$ is coheight preserving along $\Min(S')$, there is an element $x$ of $Y'$ such that $\Psi'(x) = P'$ and $\dim(S'/P')=\dim(x_{Y'}^{\uparrow})$. 
%Since $x$ is a minimal element of $Y',$ it is also a minimal element of $Y$.  
Observe that $\Psi(x) = \Psi'(x)S = P'S = P$. Now, $\dim(x_{Y'}^{\uparrow}) = \dim(x_{Y}^{\uparrow}) + 1$, and $\dim(S/P) \leq \dim(S'/P') - 1.$ Let $\dim(x^{\uparrow}_{Y'}) = \dim(S'/P') = m$. Then there is a chain in $Y'$ of the form
$$x = x_0 <_c x_1 <_c \cdots <_c x_{m - 1} <_c x_m = z.$$ Now, $x_{m - 1} = y_i$ for some $i = 1,2, \ldots ,n$. Since $\Psi'$ is a saturated embedding, we have a saturated chain of prime ideals
$$P' = \Psi'(x) \subsetneq \Psi'(x_1) \subseteq \cdots \subsetneq \Psi'(y_i) \subsetneq M'$$ of length $m$ in $S'$, and so we have a corresponding saturated chain of prime ideals
$$P = \Psi'(x)S \subsetneq \Psi'(x_1)S \subsetneq \cdots \subsetneq \Psi'(y_i)S$$ in $S$ of length $m - 1$. It follows that $\dim(S/P) \geq m - 1$, and so we have $\dim(S/P) = m - 1$. Thus,
$$\dim(S/P) = \dim(S'/P') - 1 = \dim(x^{\uparrow}_{Y'}) - 1 = \dim(x^{\uparrow}_Y).$$
%Then there is a saturated chain of prime ideals of $S'$ of the form 
%$$P = P_0 \subsetneq P_1 \subsetneq \cdots \subsetneq P_m = \Psi'(y_i)S,$$ which corresponds to a saturated chain of prime ideals of $S'$ of the form 
%$$P' = P'_0 \subsetneq P'_1 \subsetneq \cdots \subsetneq P'_m = \Psi'(z) = M'.$$ 
Therefore, $\Psi$ is coheight-preserving along $\Min(S)$.
\end{proof}

% \begin{proof}
%  Consider the set $Y_{\mathbf{y}}:=Y \cup \{\mathbf{y}\}$ and impose the following order on $Y_{\mathbf{y}}$: if $x,x' \in Y$, then $x \leq_{Y_{\mathbf{y}}} x'$ if and only if $x \leq_Y x'$. Otherwise, $x \leq_{Y_{\mathbf{y}}} \mathbf{y}$ for every $x \in Y$. It is straightforward to show that ${Y_{\mathbf{y}}} $ forms a finite partially ordered set with a unique maximal node $\mathbf{y}$. By Theorem $\ref{local quasi-excellent ring codimension preserving theorem}$, there exists a reduced local quasi-excellent ring $(S_\mathbf{y}, M_\mathbf{y})$ containing the rationals. Moreover, there exists a codimension-preserving saturated embedding along $\Min(Y)$, $\Psi_{\mathbf{y}}: Y_{\mathbf{y}} \longrightarrow \spec(S_\mathbf{y})$, which satisfies $\Psi(\Min(Y_{\mathbf{y}}))=\Min(S_{\mathbf{y}})$.

%  Let $y_1,\dots,y_n$ be the maximal nodes of $Y$ and consider the multiplicatively closed set $\mathfrak{S}=S_\mathbf{y} \setminus \{ \bigcup_{i=1}^n \Psi(y_i)\}$. We show that the localized ring at $\mathfrak{S}$, $S:=\mathfrak{S}^{-1}S_\mathbf{y}$ satisfies all desired conditions. Clearly, $S$ contains $\Q$ and it is quasi-excellent. By Lemma $\ref{poset isomorphism from saturated embedding}$, $Y_\mathbf{y}$ is poset isomorphic to $\Psi(Y_y)$ and $\Psi(Y_y)$ is a saturated subset of $\spec(S_\mathbf{y})$. 
% \end{proof}

\section{Local Domain Construction}\label{local domain construction}
Theorem~$\ref{local quasi-excellent ring codimension preserving theorem}$ is the first milestone in proving the existence of the local UFD in our main theorem (Theorem $\ref{main theorem of characterization of embeddings of UFDs}$). Given a finite poset $X$ with one maximal element and with $\dim(X) \geq 2$, we define $X^{n-1}$ to be the subposet of $X$ consisting of all nodes of $X$ of height at least one. By applying Theorem $\ref{local quasi-excellent ring codimension preserving theorem}$ to $X^{n-1}$, we obtain a reduced local quasi-excellent ring $B'$ and a saturated embedding $\phi': X^{n-1} \to \operatorname{Spec}(B')$. Furthermore, Corollary~$\ref{codimension implies dimension}$ guarantees $\phi'$ is dimension-preserving.

In this section, we construct a local domain $B$ from $B'$ and a dimension-preserving saturated embedding $\psi: X \longrightarrow \operatorname{Spec}(B)$. The key property of $B$ is that for carefully chosen prime ideals $Q$ of $B$, $Q \not\subseteq \mathfrak{q}$ for every $\mathfrak{q} \in \ass(B/bB)$ and every regular element $b \in B$. This ``avoiding" condition on prime ideals of $B$ plays an essential role for our strategy to build a local UFD $A \subseteq B$. In particular, if a prime ideal of $B$ satisfies this condition, then we are able to construct $A$ so that it contains a generating set for that prime ideal.

Our construction of $B$ ensures that whenever $x \in X$ has height at least two, the prime ideal $\psi(x)$ of $B$ satisfies the avoiding condition, guaranteeing that we can construct $A$ to include generators of all of the prime ideals corresponding to the nodes of $X$ of height two or more. Having handled these nodes, we dedicate Section~$\ref{ufd construction}$ to carefully adjoining generators corresponding to the remaining nodes of height one and showing that if $f: \spec(B) \longrightarrow \spec(A)$ is the map given by $f(P) = A \cap P$, then the map $\phi=f \circ \psi : X \longrightarrow \spec (A)$ is a saturated embedding that preserves dimension.

We begin with two auxiliary results that will be needed in the proof of the main result of this section, Theorem $\ref{local domain with desired properties}$.
\begin{lemma}\label{saturated embedding preserved by adjoining a variable}
    Let $B$ be a Noetherian ring, let $z$ be an indeterminate, let $A = B[[z]]$, and let $X$ be a poset. Suppose $\phi : X \longrightarrow \spec(B)$ is a saturated embedding. Then $\psi : X \longrightarrow \spec(A)$ given by $\psi(u) = (\phi(u),z)A$ is a saturated embedding.
\end{lemma}
\begin{proof}
    The proof follows from the proof of Lemma 6.1 in \cite{Colbert}.
\end{proof}
\begin{lemma}\label{dimension preserving saturated embedding by adjoining variable}
     Let $B$ be a Noetherian ring, let $z$ be an indeterminate, let $A = B[[z]]$, and let $X$ be a finite poset with $\dim(X) \geq 1$. Suppose that $X$ has $m$ height zero nodes and $m$ height one nodes denoted by $\{u_1,\dots,u_m\}$ and $\{x_1,\dots,x_m\}$ respectively. Moreover, suppose that  $u_i <_c x_j$ if and only if $i=j$. Let $X^{n-1}$ be the subposet of $X$ consisting of all nodes of $X$ of height at least one, and suppose that $\phi : X^{n-1} \longrightarrow \spec(B)$ is a dimension-preserving saturated embedding such that $\phi(\min  X^{n-1}) = \Min(B)$. Then, there exists a dimension-preserving saturated embedding $\Psi : X \longrightarrow \spec (A)$ such that $\Psi(\min X) = \Min(A)$ and if $u \in X^{n-1}$, then $\Psi(u)=(\phi(u),z)$.
\end{lemma}
\begin{proof}
    Let $\Min(B)=\{\phi(x_1):=P_1,\dots,\phi(x_m):=P_m\}$. By Lemma $\ref{saturated embedding preserved by adjoining a variable}$, the map $\psi: X^{n-1}\longrightarrow \spec (A)$ given by $\psi(u) = (\phi(u),z)A$ is a saturated embedding. Define $\Psi : X \longrightarrow \spec (A)$ as:
    $$u \mapsto \begin{cases}
        \psi(u) \ \text{ if } u \in X^{n-1} \\
        P_iA \ \text{ if } u=u_i \text{ for }i=1,\dots,m
    \end{cases}$$
    We claim that $\Psi$ has the desired properties. Note that since $\Min(A) = \{P_1A, \ldots ,P_mA\}$, we have $\Psi(\min X) = \Min(A)$. If $w_1,w_2 \in X$ with $w_1 \leq_X w_2$, then $\Psi(w_1) \subseteq \Psi(w_2)$. Now, suppose that $\Psi(w_1) \subseteq \Psi(w_2)$. If $w_1$ is not a height zero node in $X$, then $\Psi(w_1)=\psi(w_1)$ and so $\Psi(w_2)$ is not a minimal prime ideal of $A$. This implies $\Psi(w_2)=\psi(w_2)$. Since $\psi$ is an embedding, $w_1 \leq_{X^{n-1}} w_2$, hence $w_1 \leq_{X} w_2$. If $w_1$ is a height zero node of $X$, let $w_1=u_i$ and $\Psi(w_1)=P_iA$ for some $i=1,\dots,m$. If $\Psi(w_2)=P_jA$, then $w_1=w_2$ by construction of $\Psi$. Otherwise, assume \(w_2 \in X^{n-1}\). Then \(P_iA \subseteq \psi(w_2) = (\phi(w_2), z)\). In particular, \(\psi(x_i) = (P_i, z) \subseteq (\phi(w_2), z) = \psi(w_2)\). Because \(\psi\) is an embedding, it follows that \(x_i \leq_{X^{n-1}} w_2\). Consequently, in $X$ we have \(u_i < x_i \leq_X w_2\), implying \(w_1 = u_i \leq_X w_2\), and so $\Psi$ is a poset embedding.

    To show that $\Psi$ is a saturated embedding, we argue similarly. Suppose that $w_1 <_c w_2$ in $X$. If $w_1,w_2 \in X^{n-1}$, then $\Psi(w_1)<_c \Psi(w_2)$ since $\psi$ is a saturated embedding. If \( w_1 = u_i \) for some \( i = 1, \dots, m \), then our assumptions guarantee that \( x_i \) is the only node that covers \( u_i \), which means \( w_2 = x_i \). Since \( \Psi(x_i) = (P_i, z)A \) is a height-one prime ideal of \( A \), it follows that \( \Psi(u_i) <_c \Psi(x_i) \), establishing that \( \Psi \) is a saturated embedding. 

    Lastly, $\dim(X)=\dim(X^{n-1})+1$ and $\dim(A)=\dim(B)+1$, hence $\dim(X)=\dim(A)$ and $\Psi$ is dimension-preserving.
\end{proof}
% \begin{remark}\label{remark about psi map}
% Observe that, since \( \phi \) is a saturated embedding preserving dimension, we can write \( \Min(B) = \{\phi(x_1) := P_1B, \dots, \phi(x_m) := P_mB\} \). Consequently, the condition \( \Psi(\Min(X)) = \Min(A) \), combined with \( \Psi \) being a saturated embedding, ensures that \( \Psi(u_i) = P_iA \) for all \( i = 1, \dots, m \).

% \end{remark}
%The following is a subtle but key property about associated primes. 
%\begin{lemma}\label{equivalence condition for ass}
%    Let $R$ be a Noetherian local ring with maximal ideal $\mathfrak{m}$. Let $I\subseteq \mathfrak{m}$ be an ideal. Let $M$ be a finite $R-$module. If there exists nonzero $x \in I$ which is not a zerodivisor on $M$, then $I\not\subseteq \mathfrak{q}$ for all $\mathfrak{q} \in \ass(M)$.
%\end{lemma}
%\begin{proof}
%    This is clear since the union of associated primes of $M$ is equal to the set of zero-divisors of $M$. Thus, $I \subseteq \mathfrak{q}$ for some $\mathfrak{q} \in \ass(M)$ implies that $\mathfrak{q}$ contains a regular element, a contradiction. 
%\end{proof}
We now prove the main theorem of this section. 

\begin{theorem}\label{local domain with desired properties}
    Let $X$ be a finite poset of dimension $n \geq 2$ with a unique minimal node and a unique maximal node. Then, there exists a quasi-excellent local domain $B$ containing $\Q$ with $|B|=|B/M| = c$ and depth$(B) \geq 2$, and a dimension-preserving saturated embedding $\psi: X \longrightarrow \spec(B)$ such that if $x\in X$ with $\height(x)=1$, $Q \in \spec(B)$ with $\psi(x) \subsetneq Q$, and $b$ is an element of $B$, then $Q \not\subseteq \mathfrak{q}$ for every $\mathfrak{q} \in \ass_B(B/bB)$ and $\text{depth}(B_Q) \geq 2$. Moreover, if $x \in X$ with ht$(x) \leq 2$, then ht$(\psi(x)) = \height(x)$.
\end{theorem}
\begin{proof}
Let $X^{n-1}$ be the induced poset of $X$ given by $$X^{n-1}=\{x \in X: \height(x) \geq 1 \}$$
    where $x \leq y$ in $X^{n-1}$ if and only if $x \leq y$ in $X$.
Let $\{x_1,\dots,x_m\}$ be the set of height one nodes of $X$ (or the minimal nodes of $X^{n-1}$) and $z$ the unique minimal node of $X$. By Theorem $\ref{local quasi-excellent ring codimension preserving theorem}$, there exists a reduced quasi-excellent local ring $(B',M')$ together with a coheight-preserving saturated embedding $\phi: X^{n-1} \longrightarrow \spec(B')$ along $\Min(B')$ such that $B'$ contains $\Q$ and $|B'|=|B'/M'| = c$. By Lemma \ref{MinsgotoMins}, $\phi(\min X^{n-1})=\Min(B')$ and by Remark \ref{remark about dimension and codimension preserving}, $\phi$ is a dimension-preserving saturated embedding. Let $\Min(B')=\{\phi(x_1):=P_1,\dots,\phi(x_m):=P_m\}$. 

 Let $Y=\{(x_1,x_1),\dots,(x_m,x_m)\}$ and define the $\emph{extension}$ $X^e$ of $X^{n-1}$ as the set $X^{n-1} \cup Y$. We place an order relation on $X^e$. Suppose that $u,v \in X^e$. If $u,v \in X^{n-1}$, then $u \leq_{X^e}v$ if and only if $u \leq_{X^{n-1}}v$. If $u,v \in Y$, then define $u \leq_{X^e}v$ if and only if $u=v$. Otherwise, $u \leq_{X^e} v$ if and only if $u=(x_i,x_i)$ for some $i=1,\dots,m$ and $x_i \leq_{X^{n-1}} v$. By construction of $X^e$, we see that the number of minimal nodes equals the number of height one nodes and for $i,j=1,\dots,m$, $(x_i,x_i)<_{c} x_j$ if and only if $i = j$. Now apply Lemma $\ref{dimension preserving saturated embedding by adjoining variable}$ to obtain a dimension-preserving saturated embedding $\Psi: X^e \longrightarrow \spec(S)$ where $S=B'[[t]]$ for an indeterminate $t$ and $\Psi(\min X^{e})=\Min(S)$. By the proof of Lemma $\ref{dimension preserving saturated embedding by adjoining variable}$, $\Psi$ is explicitly given by $$u \mapsto \begin{cases}
        (\phi(u),t)S \ \text{ if } u \in X^{n-1} \\
        P_iS \ \text{ if } u=(x_i,x_i) \text{ for }i=1,\dots,m
    \end{cases}$$
 Notice that $S$ is a reduced quasi-excellent local ring with maximal ideal $M=(M',t)$ that contains the rationals. Note also that $B'/M' \cong S/M$ and $|B'| = |S|$, and so $|S| = |B'| = |B'/M'| = |S/M| = c$. 

 Suppose $y$ is a height two node of $X^e$. Then it is a height one node of $X^{n - 1}$. It follows that $\phi(y)B'$ is a height one prime ideal of $B'$, and so $(\phi(y),t)B'[t]$ is a height two prime ideal of $B'[t].$ Note that $B'[[t]]$ is a faithfully flat extension of $B'[t]$. It follows that $\Psi(y) = (\phi(y),t)B'[[t]]$ is a height two prime ideal of $S = B'[[t]].$
 
 %By definition, $\Psi(y) = (\phi(y),t)S$. Since ht$(\phi(y)) = 1$ in $B'$, we have dim$(B'_{\phi(y)}) = 1$, and so dim$(B'_{\phi(y)}[[t]]) = 2$. In the ring $B'[[t]]_{(\phi(y),t)} = S_{\Psi(y)},$ all elements of $B'$ that are not in $\phi(y)$ have been inverted since $(\phi(y),t)S \cap B' = \phi(y)$. It follows that dim$(S_{\Psi(y)}) \leq 2$, and so ht$(\Psi(y)) \leq 2$ in $S$. Now, for some $i = 1,2, \ldots ,m$, we have $P_iS \subsetneq (P_i,t)S \subseteq (\phi(y),t)S$. If $(P_i,t)S = (\phi(y),t)S$ then $P_i = (P_i,t)S \cap B' = (\phi(y),t)S \cap B' = \phi(y),$ contradicting that ht$(\phi(y)) = 1$ in $B'$. Thus, ht$(\Psi(y)) = 2$ in S.

We now show that if \( x \in \{x_1, \ldots ,x_m\} \), \( P \in \text{Spec}(S) \) with \( \Psi(x) \subsetneq P \) then $\text{depth}(S_P) \geq 2$. Moreover, if $s$ is a regular element of $S$, then $P \not\subseteq \mathfrak{p}$ for all $\mathfrak{p} \in \ass_S(S/sS)$. 
First, notice that $x \in X^{n - 1}$, and so $t \in \Psi(x)$. Now, $\Psi(x) \subsetneq P$ implies that $t \in P$ and there exists an element \( f(t) \in  P \setminus \Psi(x) \), which can be expressed as  
\[
f(t) = a_0 + g(t),
\] 
where \( a_0 \in B' \) and \( g(t) \in (t)S \). Since \( t \in \Psi(x) \), we have \( a_0 \not\in \Psi(x) \), establishing that \( \Psi(x) \cap B' \subsetneq P \cap B' \). Since \( B' \) is reduced, its associated prime ideals coincide with its minimal prime ideals, that is, \( \operatorname{Ass}_{B'}(B') = \operatorname{Min}(B') \). In particular, \( \Psi(x) \cap B' \subsetneq P \cap B' \) implies that $P\cap B' \not\in \ass_{B'}(B')$, which guarantees the existence of a nonzero regular element \( w \in P \cap B' \subseteq P\). Since $S/tS \cong B'$, $t,w$ is a regular sequence in $S$ where $t,w \in P$.  Thus, depth$(S) \geq 2$. Since the extension $S \longrightarrow S_P$ is flat, $t,w$ is a regular sequence in $S_P$ and so depth$(S_P) \geq 2$. Suppose that $\mathfrak{p}$ is a prime ideal of $S$ with $P \subseteq \mathfrak{p}$. Then $t,w \in \mathfrak{p}$ and so depth$(S_{\mathfrak{p}}) \geq 2$ as well.
%Since $S/tS \cong B'$, it follows that \(P \not\subseteq \mathfrak{p}\) for every \(\mathfrak{p} \in \ass(S/tS)\). 
Let $s$ be a regular element of $S$ and suppose that $P \subseteq \mathfrak{p}$ for some $\mathfrak{p} \in \ass_S(S/sS)$. Then $\mathfrak{p}S_{\mathfrak{p}} \in \ass_{S_{\mathfrak{p}}}(S_{\mathfrak{p}}/sS_{\mathfrak{p}}),$ and so depth$(S_{\mathfrak{p}}) = 1$, a contradiction. Hence, \(P \not\subseteq \mathfrak{p}\) for every \(\mathfrak{p} \in \ass_S(S/sS)\).

Now, since $X$ has a unique minimal node, it is a height-zero gluing of $X^e$ along $C=\min X^e$ with gluing map $g: X^e \longrightarrow X$ given by $g(x)=x$ if $x$ has positive height and $g(x)=z$ if $x \in \min X^e$. Apply Theorem $\ref{height zero gluing theorem}$ to find a reduced quasi-excellent local ring $B \subseteq S$ with maximal ideal $B \cap M$ such that $B$ contains $\Q$, $\widehat{B}=\widehat{S}$, $|B/(B \cap M)|$ is uncountable, $|B|=|B/(B \cap M)|$, and there is a saturated embedding $\psi: X \longrightarrow \spec(B)$ such that $\psi(\min X)=\Min(B)$. 
Now, $|B/(B \cap M)| = |S/M| = c$, and so we have $|B|=|B/(B \cap M)| = c$. Since $X$ has a unique minimal node and $\psi(\min X)=\Min(B)$, $B$ has a unique minimal prime ideal. As $B$ is reduced, it follows that it is an integral domain. Because $\widehat{B} = \widehat{S}$ we have depth$(B) = \mbox{depth}(S) \geq 2$ and $\dim(B) = \dim(S)$. By Lemma $\ref{gluing preserves dimension}$, $\dim(X^e)=\dim(X)$. Thus, $\psi$ is a dimension-preserving saturated embedding. It follows that if $x \in X$ with ht$(x) \leq 1$, then $\height(\psi(x)) = \height(x)$.

We now show that if $x\in X$ with $\height(x)=1$, $Q \in \spec(B)$ with $\psi(x) \subsetneq Q$ and $b$ is an element of $B$, then $Q \not\subseteq \mathfrak{q}$ for every $\mathfrak{q} \in \ass_B(B/bB)$ and $\text{depth}(B_Q) \geq 2$. Suppose that $b=0$. Then, $\ass_{B}(B/bB)=\ass_B(B)=(0)$. Since $(0) \subseteq \psi(x) \subsetneq Q$, the result follows. On the other hand, assume that $b$ is a regular element of $B$ and that $Q \subseteq \mathfrak{q}'$ for some $\mathfrak{q}' \in \ass_{B}(B/bB)$. We show that this leads to a contradiction.

By Remark $\ref{remark about properties of height zero gluing}$, the ring \( B \) is obtained by applying Theorem $\ref{gluing theorem}$ to \( S \) with \( C_1 = \operatorname{Min}(S) \). Thus, the map $f\colon \spec(S) \to \spec(B)$ given by $f(P) = P \cap B,$ is surjective, and if $P$ is a prime ideal of $S$ with positive height, then $f(P)S = P$. In particular, if \( f(P) = f(P') \) for \( P, P' \in \spec(S) \) of positive height, then \( P = P' \). Moreover, by Remark $\ref{remark about properties of height zero gluing}$, we obtain the following commutative diagram:
    \[
    \begin{tikzcd}
    \spec(S) && \spec(B) \\
    \\
    X^e && X
    \arrow["f", from=1-1, to=1-3]
    \arrow["\Psi", from=3-1, to=1-1]
    \arrow["g"', from=3-1, to=3-3]
    \arrow["\psi"', from=3-3, to=1-3]
    \end{tikzcd}
    \]

Since $f$ is surjective, there are prime ideals $I, P$ and $\mathfrak{p}'$ of $S$ such that \(\psi(x) = I \cap B\), \(Q = P \cap B\) and $\mathfrak{q}'=\mathfrak{p}' \cap B$. By Remark \ref{nicecorrespondence}, $f$ induces an inclusion-preserving bijective correspondence between $\spec(S) \setminus \Min(S)$ and $\spec(B) \setminus \Min(B)$ and so \(\psi(x) \subsetneq Q\) in $B$ implies \(I \subsetneq P\) in $S$. Furthermore, $x$ being a height one node in $X$ guarantees that $g(x)=x.$ So,
\[
\psi(x) = \psi(g(x)) = f(\Psi(x)).
\] 
Thus, 
\[
f(I) = I \cap B = \psi(x) = f(\Psi(x)),
\] 
which implies \( I = \Psi(x) \). Consequently, \(\Psi(x)\subsetneq P\). Lemma~\ref{ass and subrings, good remain good} then gives \(\mathfrak p'\in\ass_S(S/bS)\). Since \(\widehat B\cong\widehat S\), any regular element \(b\in B\) remains regular in \(S\), so our previous argument shows \(P\not\subseteq\mathfrak p'\). On the other hand, \(Q\) has positive height in \(B\), hence the extension $Q \subseteq \mathfrak{q'}$ to $S$ forces \(P\subseteq\mathfrak p'\). This contradiction shows that $Q \not\subseteq \mathfrak{q}$ for every $\mathfrak{q} \in \ass_B(B/bB)$.

We now prove depth\((B_Q)\ge2\). If depth\((B_Q)=0\), then \(Q\in\ass_B(B)=(0)\), a contradiction. If depth\((B_Q)=1\), then for some regular \(b'=\frac{c}{s}\) in \(B_Q\) one has \(QB_Q \in\ass_{B_Q}(B_Q/b'B_Q)\), whence \(Q=P\cap B\in\ass_B(B/cB)\). Again, Lemma \ref{ass and subrings, good remain good} yields \(P\in\ass_S(S/cS)\), which directly implies \(PS_P\in\ass_{S_P}(S_P/cS_P)\), and therefore depth\((S_P)=1\). But \(\Psi(x)\subsetneq P\) in \(S\) forces depth\((S_P)\ge2\) as shown earlier. Thus, depth\((B_Q)\ge2\). 

Finally, suppose $y \in X$ with ht$(y) = 2$. Then $y \in X^e$, $g(y) = y$ and $y$ has height two in $X^e$. It follows by our previous work that ht$(\Psi(y)) = 2$ in $S$. Since $f$ induces an inclusion-preserving bijection between the prime ideals of positive height in $S$ and the prime ideals of positive height in $B$, we have ht$(f(\Psi(y))) = 2$. Therefore, ht$(\psi(g(y))) = \height(\psi(y)) = 2$, and we have ht$(\psi(y)) = \height(y)$ as desired.
\end{proof}

\section{UFD Construction}\label{ufd construction}
This section is dedicated to the construction of the UFD and the dimension-preserving saturated embedding for Theorem $\ref{main theorem of characterization of embeddings of UFDs}$. We first prove lemmas analogous to those established in \cite{Bonat}, specifically within the context of so-called PN-subrings (see Definition \ref{PNsubring}).

The overall construction of our UFD $A$ proceeds as follows: Given a finite poset $X$ of dimension at least two with a unique minimal node and a unique maximal node, we start with the quasi-excellent local domain $(B, M)$ and the dimension-preserving saturated embedding $\phi: X \to \mathrm{Spec}(B)$ obtained from Theorem $\ref{local domain with desired properties}$. Denote the height one nodes of $X$ as $\{x_1, \dots, x_m\}$. Starting from $\mathbb{Q}$, we adjoin a carefully chosen transcendental element $\widetilde{x_1} \in \phi(x_1)$ over $\mathbb{Q}$, selected to respect the inclusion relations within $X$. Specifically, by ``respecting inclusion", we mean that whenever $x \in X$ and $x_1 \not\leq x$, it follows that $\widetilde{x_1} \notin \phi(x)$.  We then localize at the appropriate ideal to construct a UFD, $R_1 = \mathbb{Q}[\widetilde{x_1}]_{(\mathbb{Q}[\widetilde{x_1}] \cap M)}$. Next, we adjoin a transcendental element $\widetilde{x_2} \in \phi(x_2)$ over $R_1$ to $R_1$ and localize in a similar manner, forming $R_2$. This process is repeated iteratively for all height one nodes, producing an ascending chain of UFDs:

\[
\mathbb{Q} \subseteq R_1 \subseteq R_2 \subseteq \dots \subseteq R_m.
\]

Building on this, we adapt the UFD construction presented in \cite{Colbert} to construct a UFD $A \subseteq B$ that has $R_m$ as a subring. The ring $A$ will be constructed in a manner that ensures the inclusion of a generating set for $\phi(x)$ where $x \in X$ is a node with height greater than or equal to 2. Using this framework, we demonstrate the existence of a dimension-preserving saturated embedding $\psi$ from $X$ to $\mathrm{Spec}(A)$.

%\subsection{Adjoining Height-One Nodes of $X$}

We begin by stating the definition of a PN-subring from \cite{Colbert}, a type of subring that will play a crucial role in this section.

\begin{definition}\label{PNsubring}(Definition 4.3, \cite{Colbert}) Let $(B,M)$ be a local domain with $B/M$ uncountable, and let $(R,M \cap R)$ be an infinite quasi-local unique factorization domain contained in $B$ such that $|R| < |B/M|$ and, if $b \in B$ and $P \in \ass(B/bB)$, then $\height(P \cap R) \leq 1$. Then $R$ is called a {\em Pseudo-N-Subring of $B$}, or a {\em PN-subring of $B.$} 
\end{definition}

The following lemma is a generalization of the classical prime avoidance theorem. This particular generalization was proved by Heitmann~\cite{HeitmannUFD}, and we will rely on it several times throughout our construction of the ring~$A$.
\begin{lemma}[Lemma 3, \cite{HeitmannUFD}]\label{coset avoidance}
    Let $(B,M)$ be a local ring. Let $C \subseteq \spec(B)$, let $I$ be an ideal of $B$ such that $I \not\subseteq P$ for every $P \in C$, and let $D$ be a subset of $B$. Suppose $|C \times D| < |B/M|$. Then $I \not\subseteq \bigcup\{(P+r)|P\in C,r\in D\}$.
\end{lemma}
Next, we carefully adjoin selected transcendental elements associated with height one prime ideals of \( B \). To do so, we prove a more general version of Lemma 4.4 in \cite{Colbert}, which serves as the key result for preserving inclusion relations among finitely many height one prime ideals of \( B \) in the constructed UFD \( A \).

\begin{lemma}[Lemma 4.4, \cite{Colbert}]
    Let $(B,M)$ be a local domain with $B/M$ uncountable, and let $(R,M\cap R)$ be a PN-subring of $B$. Let $C = \{P \in \spec(B)|P \in \ass(B/rB) \text{ for some } r \in R\}$. Let $x \in B$ be such that $x + P \in B/P$ is transcendental over $R/(R \cap P)$ for every $P \in C$. Then $S = R[x]_{(R[x]\cap M)}$ is a PN-subring of $B$ with $|S| = |R|$. Moreover, prime elements in $R$ are prime in $S$.
\end{lemma}

\begin{lemma}\label{adjoining transcendetals}
    Let $(B,M)$ be a local domain with $B/M$ uncountable, and let $(R,M \cap R)$ be a PN-subring of $B$. Let $W$ be a countable set of nonmaximal ideals of $B$, and let $C = \{P \in \spec(B)|P \in \ass(B/rB) \text{ for some } r \in R\} \cup W$. Let $x \in B$ be such that $x + P \in B/P$ is transcendental over $R/(R \cap P)$ for every $P \in C$. Then $S = R[x]_{(R[x]\cap M)}$ is a PN-subring of $B$ with $|S| = |R|$. Moreover, prime elements in $R$ are prime in $S$ and if $P \in W$ with $R \cap P=(0)$, then $S \cap P=(0)$.
\end{lemma}
\begin{proof}
  The proof of Lemma 4.4 in \cite{Colbert} establishes that $S$ is a PN-subring of $B$ such that $|S| = |R|$ and prime elements in $R$ are prime in $S$. To show the last assertion, we include the argument shown in the proof of Lemma 11 in \cite{LoeppJoAl}. Fix $P \in W$ and assume that $R \cap P=(0)$. To show that $S \cap P=(0)$, it suffices to prove that $R[x] \cap P=(0)$. Let $f(x)=a_0+a_1x+\dots +a_nx^n \in P$ with $a_i \in R$. Then, $f(x) \equiv 0$ mod $P$. Since $x+P$ is transcendental over $R/(R \cap P)$, we must have that for each $i=1,\dots, n$, $a_i \in R \cap P=(0)$. Therefore, $f(x)$ is the zero polynomial, and $R[x] \cap P=(0)$.
\end{proof}

In Lemma~$\ref{bonet's adjoining}$ we generalize Lemma 2.10 from~\cite{Bonat}. In particular, we show that given a local domain $B$ satisfying some mild conditions and a countable set of nonmaximal prime ideals of $B$, there exists an element $\widetilde{x} \in B$ that is in a select number of height one prime ideals of $B$ and that is not in any ideal of our countable set of prime ideals.  Moreover, adjoining this element to a PN-subring yields a new PN-subring. This result is essential for constructing a PN-subring of $B$ containing height one prime ideals that correctly reflect the inclusion relations of our finite poset $X$.

\begin{lemma}\label{bonet's adjoining}
    Let $(B,M)$ be a local domain with depth at least two and $B/M$ uncountable. Let $R$ be a PN-subring of $B$. Let $W$ be a countable set of nonmaximal prime ideals of $B$ and let $C = \{P \in \spec(B) \ | \ P \in \ass(B/rB) \text{ for some } r \in R\} \cup W$. Let $P_1,\dots,P_s$ be height one prime ideals of $B$ such that, for every $i = 1,2,\dots,s$ and for every $P \in C$, $P_i \not\subseteq P$. Then, there exists $\widetilde{x}\in \Pi_{i=1}^s P_i$ with  $\widetilde{x}\not\in \bigcup_{Q \in W}Q$ such that $S = R[\widetilde{x}]_{(M\cap R[\widetilde{x}])}$ is a PN-subring of $B$ satisfying:
    \begin{enumerate}[(i)]
        \item $|S| = |R|$,
        \item $S \cap P_i = \widetilde{x}S$ for every $i = 1,2,\dots,s$,
        \item If $Q \in W$ such that $R \cap Q = (0)$, then $S \cap Q = (0)$ and,
        \item prime elements in $R$ are prime in $S$.
    \end{enumerate}
\end{lemma}
\begin{proof}
    The proof closely follows the proof of Lemma 2.10 in \cite{Bonat}. Note that $M \not\in C$ because $W$ contains non-maximal prime ideals and depth$(B) \geq 2$. Moreover, since $R$ is a PN-subring, $|R|<|B/M|$ and so $|C|<|B/M|$. For each $i=1,\dots,s$, apply Lemma $\ref{coset avoidance}$ with $D=\{0\}$ to find $x_i \in P_i$ such that $x_i \notin P$ for every $P \in C$. Define $x=\Pi_{i=1}^s x_i$ and note that $x \in \Pi_{i=1}^s P_i$ and $x \not\in \bigcup_{P \in C}P$. 

    Fix $P \in C$ and let $t,t' \in B$. If $x(1+t)+P=x(1+t')+P$ as elements of $T/P$, then $t-t' \in P$ since $x \not\in P$. Thus, $t+P=t'+P$, and it follows that if $t+P \neq t'+P$, then $x(1+t)+P\neq x(1+t')+P$. Let $D_{(P)}$ be a full set of coset representatives for the cosets $t + P$ that make $x(1 + t) + P$ algebraic over $R/(P \cap R)$, and note that $|D_{(P)}| \leq |R|<|B/M|$. Let $D = \bigcup_{P \in C} D_{(P)}$ and apply Lemma $\ref{coset avoidance}$ to find $\alpha \in M$ such that $x(1 + \alpha) + P \in B/P$ is transcendental over $R/(P \cap R)$ for every $P \in C$. Define $\widetilde{x}=x(1 + \alpha)$. By Lemma $\ref{adjoining transcendetals}$, $S=R[\widetilde{x}]_{(R[\widetilde{x}] \cap M)}$ is a PN-subring of $B$, $|S| = |R|$, prime elements in $R$ are prime in $S$, and if $Q \in W$ with $R \cap Q = (0)$, then $S \cap Q = (0)$. Note that $\widetilde{x} \not\in \bigcup_{Q \in W}Q$ because $1 + \alpha$ is a unit in $B$ and $x \not\in \bigcup_{Q \in W}Q$. 

    Fix $i=1,\dots,s$ and note that $\widetilde{x}S \subseteq S \cap P_i$. Since $R$ is a domain and $\widetilde{x}$ is transcendental over $R$, $\widetilde{x}S$ is a prime ideal of $S$. Now, $P_i$ is a height one prime ideal of $B$ containing $\widetilde{x}$. So, $P_i \in \ass(B/\widetilde{x}B)$. Since $S$ is a PN-subring of $B$, it follows that $\height(P_i \cap S)=1$. Thus, $S \cap P_i=\widetilde{x}S$.
\end{proof}
We now show that given a saturated embedding $\psi: X \longrightarrow \spec(B)$, there exists a PN-subring $R$ of $B$ such that $\psi(x) \cap R$ is principal for every height one node $x \in X$ and the inclusion relations of height one nodes of $X$ hold in $R$.

\begin{lemma}\label{PN subring with the right height ones}
Let \(X\) be a finite poset of dimension \(n \geq 2\) with a unique minimal node and a unique maximal node. Let \( (B, M) \) be a local domain containing $\Q$ with $B/M$ uncountable and depth$(B)\geq2$. Suppose that \( \psi: X \to \operatorname{Spec}(B) \) a dimension-preserving saturated embedding such that if $x \in X$ with $\height(x)=1$ and $Q \in \spec(B)$ with $\psi(x) \subsetneq Q$, then $\text{depth}(B_Q) \geq 2$. Denote by \(\{x_1, \dots, x_m\}\) the set of height one elements in \(X\). Then, there exists a countable PN-subring \(R\) of \(B\) such that:
    \begin{enumerate}[(i)]
        \item There exist distinct nonzero elements $\widetilde{x_1},\dots,\widetilde{x_m}$ of $R$ such that $\psi(x_i) \cap R=\widetilde{x_i}R$ for every $i=1,\dots,m$ and,
        \item\label{condition 2 pn subring} If $x,y \in X$ with $\height(x)=1$, then $x \leq y$ if and only if $\psi(x) \cap R \subseteq \psi(y) \cap R$.
    \end{enumerate}
\end{lemma}
\begin{proof}

 We define a family of PN-subrings recursively as follows. Let \( R_0 = \mathbb{Q} \) and note that $R_0$ is a PN-subring of $B$. Define $W_1= \{\psi(x) \in \operatorname{Spec}(B) \mid \psi(x_1) \not\subseteq \psi(x)\} \), and set \( C_1 = \{ P \in \operatorname{Spec}(B) \mid P \in \operatorname{Ass}(B / rB) \text{ for some } r \in R_0 \} \cup W_1 = \{(0)\}  \cup W_1 \). Observe that since $X$ is finite, $W_1$ is a finite set, $M \not\in C_1$, 
 %Moreover, for \( j = 1, \dots, m \), the containment \( \psi(x_1) \subseteq \psi(x_j) \) holds if and only if \( j = 1 \). 
 and \( \psi(x_1) \not\subseteq P \) for every \( P \in C_1 \). Since $\psi$ is a dimension-preserving saturated embedding, $\psi(x_1)$ is a height one prime ideal of $B$. By Lemma \(\ref{bonet's adjoining}\), there exists \( \widetilde{x_1} \in \psi(x_1) \) such that \( \widetilde{x_1} \not\in \bigcup_{Q \in W_1} Q \), and the ring \( R_1 = R_0[\widetilde{x_1}]_{(R_0[\widetilde{x_1}] \cap M)} \) is a countable PN-subring of $B$ satisfying the conditions: \( R_1 \cap \psi(x_1) = \widetilde{x_1} R_1 \),  prime elements in \( R_0 \) remain prime in \( R_1 \), and if $Q \in W_1$ with $R_0 \cap Q = (0)$ then $R_1 \cap Q = (0)$. In particular, if $i \geq 2$ then $\psi(x_i) \in W_1$ and $R_0 \cap \psi(x_i) = (0)$ and so $R_1 \cap \psi(x_i) = (0)$.

Now, define $W_2= \{\psi(x) \in \operatorname{Spec}(B) \mid \psi(x_2) \not\subseteq \psi(x)\} \) and 
$ C_2 = \{ P \in \operatorname{Spec}(B) \mid P \in \operatorname{Ass}(B/rB)\text{ for some } r \in R_1 \} \cup W_2.$ Once again, \( W_2 \) is a finite set of nonmaximal prime ideals of $B$. We claim that \( \psi(x_2) \not\subseteq P \) for every \( P \in C_2 \). By definition, $\psi(x_2) \not\subseteq P$ for every $P \in W_2$. Suppose that \( \psi(x_2) \subseteq Q \) for some \( Q \in \operatorname{Ass}(B/rB) \) with \( r \in R_1 \). Since \( \psi(x_2) \cap R_1 = (0) \) and $r \in R \cap Q,$ $\psi(x_2) \neq Q$. By hypothesis, $\psi(x_2) \subsetneq Q$ implies that \( \operatorname{depth}(B_Q) \geq 2 \). However, since \( Q \in \operatorname{Ass}(B/rB) \), we have \( QB_Q \in \operatorname{Ass}(B_Q / rB_Q) \) and \( \operatorname{depth}(B_Q) = 1 \), a contradiction. Thus, we conclude that \( \psi(x_2) \not\subseteq P \) for every \( P \in C_2 \). By Lemma \(\ref{bonet's adjoining}\), there exists an element \( \widetilde{x_2} \in \psi(x_2) \) such that \( \widetilde{x_2} \not\in \bigcup_{Q \in  W_2} Q \) and the ring  $R_2 = R_1[\widetilde{x_2}]_{(R_1[\widetilde{x_2}] \cap M)}$ is a countable PN-subring satisfying the conditions \( R_2 \cap \psi(x_2) = \widetilde{x_2} R_2 \),  prime elements in \( R_1 \) remain prime in \( R_2 \), and if $Q \in W_2$ with $R_1 \cap Q = (0)$ then $R_2 \cap Q = (0).$ In particular, if $i \geq 3$ then $\psi(x_i) \in W_2$ and $R_1 \cap \psi(x_i) = (0)$ and so $R_2 \cap \psi(x_i) = (0)$. Now, $\psi(x_1) \in W_2$ and so $\widetilde{x_2} \not\in \psi(x_1)$, and it follows that $\widetilde{x_1} \neq \widetilde{x_2}$. We claim that $R_2 \cap \psi(x_1) = \widetilde{x_1}R_2.$ Note that $\widetilde{x_1}R_2 \subseteq R_2 \cap \psi(x_1)$. Since $R_1$ is a domain and $\widetilde{x_1}$ is transcendental over $R_1$, $\widetilde{x_1}R_2$ is a prime ideal of $R_2$. Since $\psi(x_1)$ is a height one prime ideal of $B$ containing $\widetilde{x_1}$, we have that $\psi(x_1) \in \ass(B/\widetilde{x_1}B).$ As $R_2$ is a PN-subring of $B$, it follows that ht$(R_2 \cap \psi(x_1)) = 1$ and so $R_2 \cap \psi(x_1) = \widetilde{x_1}R_2.$

We continue this process by defining  $W_i= \{\psi(x) \in \operatorname{Spec}(B) \mid \psi(x_i) \not\subseteq \psi(x)\},$ and  setting \(C_i = \{ P \in \operatorname{Spec}(B) \mid P \in \operatorname{Ass}(B/rB) \text{ for some } r \in R_{i-1} \} \cup W_i\) for $i=3,\dots,m$. At each step, an element \( \widetilde{x_i} \in \psi(x_i) \) is chosen such that \( \widetilde{x_i} \not\in \bigcup_{Q \in W_i} Q \). The subsequent countable PN-subring is then defined as $R_{i}=R_{i-1}[\widetilde{x_{i}}]_{(R_{i-1}[\widetilde{x_{i}}] \cap M)}$. This ensures that for each $i$, \( R_i \cap \psi(x_i) = \widetilde{x_i} R_i \), prime elements of \( R_{i-1} \) remain prime in $R_i$, and $R_{i} \cap \psi(x_{j}) = (0)$ for $j > i$. Moreover, $R_i \cap \psi(x_k) = \widetilde{x_k}R_i$ for all $k < i$ and $\widetilde{x_1}, \widetilde{x_2}, \ldots ,\widetilde{x_i}$ are distinct. This iterative construction produces an ascending chain of PN-subrings of $B$ $$R_0 \subseteq R_1 \subseteq R_2 \subseteq \dots \subseteq R_m.$$

Define \( R = R_m \). By construction, $\widetilde{x_1}, \widetilde{x_2}, \ldots ,\widetilde{x_m}$ are distinct and \( \psi(x_i) \cap R = \widetilde{x_i}R \) for all \( i = 1, \dots, m \). Now, consider \( x, y \in X \) with \( \operatorname{ht}(x) = 1 \). If \( x \leq y \), then \( \psi(x) \subseteq \psi(y) \), as $\psi$ is a saturated embedding, implying \( \psi(x) \cap R \subseteq \psi(y) \cap R \). To establish the ``if" direction of condition (ii), we proceed by showing its contrapositive. Let $x=x_j$ for some $j=1,\dots,m$ and suppose $x_j \not\leq y$. Then, $\psi(x_j) \not\subseteq \psi(y)$ and so $\psi(y) \in W_j$. Since \( \widetilde{x_j} \) was selected to satisfy \( \widetilde{x_j} \not\in \bigcup_{Q \in W_j} Q \), it follows that \( \widetilde{x_j}R = \psi(x_j) \cap R \not\subseteq \psi(y) \cap R \).
\end{proof}
%\begin{remark}\label{remark about avoiding prime ideals for height ones}
%In the setting of Theorem \ref{PN subring with the right height ones}, suppose that $1 \leq i \leq m$ and $y \in X$ with $x_i \not\leq y$. Then $\widetilde{x_i}$ was chosen so that $\widetilde{x_i} \not\in \psi(y).$ 

%Note that by condition $\ref{condition 2 pn subring}$ and $\psi$ being a saturated embedding, we have $\widetilde{x_i} \notin \psi(y)$ for every $i=1,\dots,m$ and for every $y \in X$ satisfying $x \not\le y$.

%\end{remark}
%\subsection{UFD Construction}
From this point onward, the construction of our desired local UFD $A$ closely follows the construction presented in Section 4 in~\cite{Colbert}. The key distinction between the UFD theorem from \cite{Colbert} (Theorem 4.13) and the main result of this section, Theorem~$\ref{UFD Theorem together with a saturated embedding}$, lies in the initial PN-subring used: in \cite{Colbert}, the local UFD is constructed as an infinite union of PN-subrings starting with the subring $\mathbb{Q}$, whereas our construction begins instead with the ring obtained in Lemma~$\ref{PN subring with the right height ones}$. Choosing this ring as our starting point ensures that the inclusion relations for the height one nodes of $X$ are properly respected within $A$.

The next result from \cite{Colbert} shows that an infinite union of PN-subrings behaves well. We begin with a definition.
\begin{definition}[Definition 4.9, \cite{Colbert}]
    Let $\Psi$ be a well-ordered set, and let $\alpha \in  \Psi$. Define $\gamma(\alpha) = \text{sup}\{\beta \in \Psi \mid \beta < \alpha\}$.
\end{definition}

\begin{lemma}[Lemma 4.10, \cite{Colbert}]\label{union of PN subrings}
    Let $(B,M)$ be a local domain with $B/M$ uncountable, and let $R_0$ be a PN-subring of $B$. Let $\Omega$ be a well-ordered set with least element 0, and assume that for every $\alpha \in \Omega$, $|\{\beta \in \Psi \mid \beta < \alpha\}| < |B/M|$. Suppose $\{R_\alpha \mid \alpha \in \Omega\}$ is an ascending collection of rings such that if $\gamma(\alpha)=\alpha$, then $R_\alpha=\bigcup_{\beta<\alpha} R_\beta$ while if $\gamma(\alpha)<\alpha$, $R_\alpha$ is a PN-subring of $B$ with $R_{\gamma(\alpha)} \subseteq R_\alpha$ and prime elements in $R_{\gamma(\alpha)}$ are prime in $R_\alpha$.

    Then $S = \bigcup_{\alpha \in \Omega} R_\alpha$ satisfies all conditions to be a PN-subring of $B$ except for possibly the condition that $|S| < |B/M|$. Moreover, $|S| \leq \text{sup}(|R_0|,|\Omega|)$ and elements that are prime in some $R_\alpha$ are prime in $S$.
\end{lemma}
To preserve dimension, we construct our UFD $A$ to be a subring of the local domain $B$ obtained from Theorem $\ref{local domain with desired properties}$ such that $\widehat{A} \cong \widehat{B}$. The next proposition gives sufficient conditions for this.
\begin{prop}[Proposition 2.6, \cite{gluingpaper}]\label{completion proving machine}
    Let $(B,M)$ be a local ring, and let $\widehat{B}=T$. Suppose $(S,S\cap M)$ is a quasi-local subring of $B$ such that the map $S\longrightarrow B/M^2$ is onto and $IB\cap S=I$ for every finitely generated ideal $I$ of $S$. Then $S$ is Noetherian and $\widehat{S} = T$. Moreover, if $B/M$ is uncountable and $|B| = |B/M|$, then $S/(S \cap M)$ is uncountable, and $|S| = |S/(S \cap M)|$.
\end{prop}
In view of Proposition \ref{completion proving machine}, we use the following lemma to show that $\widehat{A} \cong \widehat{B}$. In addition, similar to the approach in the proof of Theorem 4.13 from \cite{Colbert}, we use the lemma to adjoin generating sets of carefully selected prime ideals of $B$ to $A$.
%is essential to ensure that the ring $B$ and our constructed UFD $A$ share the same completion. Similarly to the approach from Theorem 4.13 in~\cite{Colbert}, we exploit the faithfully flatness of $T$ as a $B-$module to adjoin carefully selected generating sets of prime ideals. 

\begin{lemma}[Lemma 4.11, \cite{Colbert}]\label{ideals are closed up}
    Let $(B,M)$ be a local domain with $\text{depth}(B) \geq 2$ and with $B/M$ uncountable, and let $(R, M\cap R)$ be a PN-subring of $B$. Suppose $Q$ is a prime ideal of $B$ such that if $P \in \spec(B)$ with $P \in \ass(B/bB)$ for some $b\in B$, then $Q \not\subseteq P$. Let $u\in B$. Then there is a PN-subring $(S,M\cap S)$ of $B$ such that $R \subseteq S$, $|S| = |R|$, prime elements in $R$ are prime in $S$, $S$ contains an element of the coset $u + M^2$, $S$ contains a generating set for $Q$, and, for every finitely generated ideal $I$ of $S$, $IB \cap S = I$.
\end{lemma}

Our next goal is to show that for a poset $X$ satisfying our necessary conditions (see the conditions in Theorem~$\ref{main theorem of characterization of embeddings of UFDs}$), there exists a local UFD $A$ and a dimension-preserving saturated embedding $\phi: X \longrightarrow \spec(A)$. In Theorem~$\ref{UFD theorem}$ we prove that, given a local domain $B$ satisfying some mild conditions and a dimension-preservng saturated embedding from $X$ to $\Spec(B)$ also satisfying some mild conditions, there exists a local UFD $A$ with nice geometric and algebraic properties such as $\widehat{A}=\widehat{B}$ (this ensures, for example, that $\dim(A)=\dim(X)$). Then, in Theorem~$\ref{UFD Theorem together with a saturated embedding}$, we define $\phi$ and exploit properties of the prime spectrum of $A$ to show that $\phi$ is a dimension-preserving saturated embedding.
\begin{theorem}\label{UFD theorem}
   Let \(X\) be a finite poset of dimension \(n \geq 2\) with a unique minimal element and a unique maximal element. Let \( (B, M) \) be a local domain containing $\Q$ with $B/M$ uncountable, $|B|=|B/M|$, and depth$(B)\geq2$. Suppose that \( \psi: X \to \operatorname{Spec}(B) \) is a dimension-preserving saturated embedding such that if $x \in X$ with $\height(x)=1$, $Q \in \spec(B)$ with $\psi(x) \subsetneq Q$ and $b$ is an element of $B$, then $Q \not\subseteq \mathfrak{q}$ for every $\mathfrak{q} \in \ass(B/bB)$ and $\text{depth}(B_Q) \geq 2$. Denote by \(\{x_1, \dots, x_m\}\) the set of height one elements in \(X\).  Then, there exists a local UFD \((A, M \cap A)\) such that:
    \begin{enumerate}
        \item $A\subseteq B$,
        \item $\widehat{A}=\widehat{B}=T,$
        \item There exist distinct nonzero elements $\widetilde{x_1},\dots,\widetilde{x_m}$ of $A$ such that $\psi(x_i) \cap A=\widetilde{x_i}A$ for every  $i=1,\dots,m$, 
        \item If $x,y \in X$ with $\height(x)=1$, then $x \leq y$ if and only if $\psi(x) \cap A \subseteq \psi(y) \cap A$ and,
        \item The map $f: \spec(B)\longrightarrow \spec(A)$ given by $f(P)=A \cap P$ is onto and if $P \in \spec(B)$ with $\psi(x) \subsetneq P$ for $x$ a height one node of $X$, then $f(P)B=P$. In particular, if $P_1$ and $P_2$ are prime ideals of $B$ such that $\psi(x) \subsetneq P_i$ for $i=1,2$, then $f(P_1)=f(P_2)$ implies that $P_1=P_2$.  
    \end{enumerate}
    
\end{theorem}
\begin{proof}
     By Lemma $\ref{PN subring with the right height ones}$, there exists a PN-subring $R$ of $B$ and distinct nonzero elements $\widetilde{x_1},\dots,\widetilde{x_m}$ of $R$ such that $\psi(x_i) \cap R=\widetilde{x_i}R$ for all $i=1,\dots, m$, and if $x,y \in X$ with $\height(x)=1$, then $x \leq y$ if and only if $\psi(x) \cap R \subseteq \psi(y) \cap R$. 

    We now closely follow the proofs of Lemma 4.12 and Theorem 4.13 in \cite{Colbert}. Set $R_0=R$. Let \( \Omega_1 = B/M^2 \) and define  
\[
\Omega_2 = \{ Q \in \operatorname{Spec}(B) \mid \text{if } b \in B \text{ and } P \in \operatorname{Ass}_B(B/bB), \text{ then } Q \not\subseteq P \}
\]  
Observe that \( |\Omega_1| = |B/M^2| = |B/M| \) and \( |\Omega_2| \leq |B| = |B/M| \). Define \( \Omega = \Omega_1 \times \Omega_2 \) and note that \( |\Omega| = |B/M| \). Well-order \( \Omega \) using an index set \( \Psi \) such that the first element of \( \Psi \) is 0 and each element of \( \Omega \) has fewer than \( |\Omega| \) predecessors. For \( \alpha \in \Psi \), denote the corresponding element in \( \Omega \) by \( (b_\alpha + M^2, Q_\alpha) \).

We recursively define a family of PN-subrings \( \{ R_\alpha \mid \alpha \in \Psi \} \). The subring \( R_0 \) has already been defined. Suppose \( \alpha \in \Psi \) and assume that \( R_\beta \) has been constructed for all \( \beta < \alpha \). If \( \gamma(\alpha) < \alpha \), define \( R_\alpha \) to be the PN-subring obtained from Lemma $\ref{ideals are closed up}$ such that \( R_{\gamma(\alpha)} \subseteq R_\alpha \), \( |R_\alpha| = |R_{\gamma(\alpha)}| \), prime elements of \( R_{\gamma(\alpha)} \) remain prime in \( R_\alpha \), \( R_\alpha \) contains an element of the coset \( b_{\gamma(\alpha)} + M^2 \), \( R_\alpha \) contains a generating set for \( Q_{\gamma(\alpha)} \), and for every finitely generated ideal \( I \) of \( R_\alpha \), we have \( I B \cap R_\alpha = I \). If \( \gamma(\alpha) = \alpha \), define \( R_\alpha = \bigcup_{\beta < \alpha} R_\beta \).  

Let \( A = \bigcup_{\alpha \in \Psi} R_\alpha \). By Lemma $\ref{union of PN subrings}$, the ring \( (A, M \cap A) \) satisfies all the conditions for being a PN-subring, except for the condition \( |A| < |B/M| \). In particular, \( A \) is a UFD and the map $A \longrightarrow B/M^2$ is onto. 

We now verify that \( I B \cap A = I \) for every finitely generated ideal \( I \) of \( A \). Let \( I = (a_1, a_2, \dots, a_n) \) be a finitely generated ideal of \( A \) and suppose \( c \in I B \cap A \). Then there exists \( \alpha \in \Psi \) such that \( c, a_1, a_2, \dots, a_n \in R_\alpha \) and for every finitely generated ideal \( J \) of \( R_\alpha \), it holds that \( J B \cap R_\alpha = J \). Taking \( J = (a_1, a_2, \dots, a_n) R_\alpha \), we get  
\[
c \in (a_1, a_2, \dots, a_n) B \cap R_\alpha = J \subseteq J A = I,
\]  
which implies \( I B \cap A = I \).  

From Proposition $\ref{completion proving machine}$, it follows that $(A, M \cap A)$ is Noetherian, and thus a local UFD, with its completion given by $\widehat{A}=\widehat{B}=T$. Since $\psi$ is a dimension-preserving saturated embedding, $\psi(x_i)$ is a height one prime ideal of $B$ for all $i = 1,2, \ldots ,m$. Now, $\psi(x_i) \cap A \neq (0)$ and so $\psi(x_i) \cap A$ is a height one prime ideal of $A$. By construction, $\widetilde{x_i}$ is a prime element of $A$ contained in $\psi(x_i)$ and so $\psi(x_i) \cap A = \widetilde{x_i}A$ for all $i = 1,2 \ldots ,m$.

Suppose $x, y \in X$ with $\height(x) = 1$ and $x \leq y$.  
Since $\psi$ is a saturated embedding, we obtain 
$\psi(x) \cap A \;\subseteq\; \psi(y) \cap A.
$ For the reverse direction, we show the contrapositive. Suppose $x, y \in X$ with $\height(x) = 1$ and $x \not\leq y$. Then $\psi(x) \cap R \not\subseteq \psi(y) \cap R$ and so there exists $r \in \psi(x) \cap R$ with $r \not\in \psi(y) \cap R.$ Note that $r \not\in \psi(y)$, and so we have $r \in \psi(x) \cap A$ with $r \not\in \psi(y) \cap A.$ Thus $\psi(x) \cap A \not\subseteq \psi(y) \cap A$.

%Let $x = x_i$ for some 
%$i \in \{1,\dots,m\}$ and assume $x_i \not\leq y$. 

%By construction, $\psi(x_i) \cap A = \widetilde{x_i}A$ for some prime element $\widetilde{x_i}$ of $R_0$ and this prime in $A$. From Remark \ref{remark about avoiding prime ideals for height ones}, we know that $\widetilde{x_i} \notin \psi(y)$ for every $y \in X$ satisfying $x_i \not\leq y$. Consequently, $\psi(x) \cap A \;=\; \widetilde{x_i}A \;\not\subseteq\; \psi(y) \cap A.
%$

Lastly, let \(J \in \operatorname{Spec}(A)\). There is a prime ideal \(P\) of \(T\) such that \(P \cap A = J\). Since \((P \cap B) \cap A = P \cap A = J\), the map \(f\) is surjective. Now suppose \(Q \in \operatorname{Spec}(B)\) satisfies $\psi(x) \subsetneq Q$ for a height one node $x \in X$. Then, by assumption, for every $b \in B$ and $P \in \ass_B(B/bB)$, $Q \not\subseteq P$. So, by construction, \(A\) contains a generating set for \(Q\), so \((A \cap Q)B = Q\). Therefore, if \(P_1, P_2 \in \operatorname{Spec}(B)\) satisfy $\psi(x) \subsetneq P_i$ for \(i=1,2\) and \(f(P_1)=f(P_2)\), then \(P_1 = (A \cap P_1)B = (A \cap P_2)B = P_2\).\end{proof}

% \begin{remark}\label{A contains a generating set for $Q$}
%     The construction of $A$ in Theorem $\ref{UFD theorem}$ ensures that if $Q \in \spec(B)$ with $\text{depth}(B_Q) \geq 2$, then $A$ contains a generating set for $Q$. 
    
%     % In particular, Theorem $\ref{UFD theorem}$ establishes an inclusion preserving bijective correspondence between prime ideals of $A$ containing $\psi(x) \cap A$ and prime ideals of $B$ containing $\psi(x)$ for all $x \in X$ with $\Ht(x)=1$.
% \end{remark}

We end this section by establishing that the conditions given in Theorem $\ref{main theorem of characterization of embeddings of UFDs}$ are indeed sufficient.
\begin{theorem}\label{UFD Theorem together with a saturated embedding}
    Let \(X\) be a finite poset of dimension at least two with a unique minimal element and a unique maximal element. Moreover, suppose that if $x,y \in X$ with $\height(x)=1$ and $x<_c y$, then $\height(y)=2$. Then, there exists a local UFD $A$ and a dimension-preserving saturated embedding $\phi: X \longrightarrow \spec(A)$.
\end{theorem}
\begin{proof}
Let \( (B, M) \) be the local domain together with the dimension-preserving saturated embedding \( \psi: X \to \operatorname{Spec}(B) \)  obtained from Theorem \(\ref{local domain with desired properties}\). Then $B$ contains the rationals, $|B| = |B/M| = c$, and depth$(B) \geq 2$. Moreover, if $x \in X$ with ht$(x) = 1$, $Q \in \Spec(B)$ with $\psi(x) \subsetneq Q$, and $b$ is an element of $B$, then $Q \not\subseteq \mathfrak{q}$ for every $\mathfrak{q} \in \ass(B/bB)$ and depth$(B_Q) \geq 2$. In addition, if $x \in X$ with ht$(x) \leq 2$ then ht$(\psi(x)) = \height(x)$.
Denote by $\{x_1,\dots,x_m\}$ the set of height one nodes in $X$. By Theorem $\ref{UFD theorem}$, there exists a local UFD satisfying the five conditions stated in Theorem \ref{UFD theorem}.
%$A \subseteq B$ such that $\widehat{A}=\widehat{B}=T$. Moreover, the map $f: \spec(B)\longrightarrow \spec(A)$ given by $f(P)=A \cap P$ is onto. It also satisfies the following: if $P \in \spec(B)$ with $\psi(x) \subsetneq P$ for $x$ a height-one node of $X$, then $f(P)B=P$ and, if $P_1$ and $P_2$ are prime ideals of $B$ such that $\psi(x) \subsetneq P_i$ for $i=1,2$, then $f(P_1)=f(P_2)$ implies that $P_1=P_2$.

Define $\phi: X \longrightarrow \spec(A)$ as $\phi=f \circ \psi$ where $f : \Spec(B) \longrightarrow \Spec(A)$ is given by $f(P) = A \cap P$. We claim that $\phi$ is a dimension-preserving saturated embedding. Let $x,y \in X$ with $x \leq y$. Then $\psi(x) \subseteq \psi(y)$ and so $\phi(x)=A \cap \psi(x) \subseteq A \cap \psi(y)=\phi(y)$. Now, suppose that $\phi(x) \subseteq \phi(y)$. Then $A \cap \psi(x) \subseteq A \cap \psi(y)$. Suppose that $A \cap \psi(x) = (0)$ and ht$(x) > 0$. Then $x_i \leq x$ for some $i = 1,2, \ldots ,m$, and so $\psi(x_i) \subseteq \psi(x).$ It follows that $\widetilde{x_i}A = A \cap \psi(x_i) \subseteq A \cap \psi(x) = (0)$, a contradiction. Hence, if $A \cap \psi(x) = (0)$ then $x$ is the minimal node of $X$ and so $x \leq y$. Now suppose $A \cap \psi(x) \neq (0)$. Then $\psi(x) \neq (0)$, and so $x$ is not the minimal node of $X$. It follows that $x_i \leq x$ for some $i = 1,2, \ldots ,m$. If $x = x_i$, then the fourth condition of Theorem \ref{UFD theorem} implies that $x \leq y.$ Now suppose $x_i \lneq x$. Then we have $\psi(x_i) \subsetneq \psi(x),$ and, in particular, ht$(\psi(x)) \geq 2$. By the fifth condition of Theorem \ref{UFD theorem}, $(A \cap \psi(x))B = \psi(x).$ Note that $\psi(y) \neq (0)$ and so $y$ is not the minimal node of $X.$ So $x_j \leq y$ for some $j = 1,2, \ldots ,m$. If $x_j = y$, then $(0) \neq A \cap \psi(x) \subseteq A \cap \psi(y) = \widetilde{x_j}A$. It follows that $A \cap \psi(x) = A \cap \psi(y) = \widetilde{x_j}A$ and so $\widetilde{x_j}B = \psi(x)$ contradicting that ht$(\psi(x)) \geq 2$. Therefore, $x_j \lneq y$ and we have $\psi(x_j) \subsetneq \psi(y)$. It follows that $(A \cap \psi(y))B = \psi(y).$ We now have $\psi(x) = (A \cap \psi(x))B \subseteq (A \cap \psi(y))B = \psi(y)$. Since $\psi$ is a saturated embedding, $x \leq y$. In all cases, we have $x \leq y$ and so $\phi$ is a poset embedding.

%If \(\height(x) \geq 2\), then there exist height one nodes \(x_i\) and $x_j$ for some \(i,j \in 1, \dots, m\) such that \(x_i \lneq x\) and \(x_j \lneq y\). Since \(\psi\) is a saturated embedding, it follows that \(\psi(x_i) \subsetneq \psi(x)\) and \(\psi(x_i) \subsetneq \psi(y)\).  Thus, the properties of the map $f$ ensures that \(\phi(x) \subseteq \phi(y)\) implies \(\psi(x) = \phi(x)B \subseteq \phi(y)B = \psi(y)\). Since \(\psi\) is an embedding, we conclude \(x \leq y\). If \(\height(x) = 1\), then by Theorem \(\ref{UFD theorem}\), \(\phi(x) = A \cap \psi(x) \subseteq A \cap \psi(y) = \phi(y)\) holds if and only if \(x \leq y\). Finally, if \(\height(x) = 0\), then \(x\) is the unique minimal node of \(X\), which implies $x \leq y$. 

To show that \(\phi\) is a saturated embedding, let \(x, y \in X\) with \(x <_c y\). Since $\phi$ is a poset embedding, $\phi(x) \subsetneq \phi(y)$ and so it remains to show that $\phi(x) \subsetneq \phi(y)$ is saturated. Suppose that there is a prime ideal $P$ of $A$ such that $\phi(x) \subseteq P \subseteq \phi(y)$. Since \(\psi\) is a saturated embedding, \(x <_c y\) implies that \(\psi(x) \subsetneq \psi(y)\) is saturated in $B$. Consider the case where \(\height(x) \geq 2\).  Then, there exists a height one node $x_i$ such that $x_i \lneq x$, which implies \(\psi(x_i) \subsetneq \psi(x)\). Thus, $(A \cap \psi(x))B = \psi(x)$. By the going down theorem, there are prime ideals $Q$ and $J$ of $B$ such that $J \subseteq Q \subseteq \psi(y)$ and $A \cap Q = P$ and $A \cap J = \phi(x)$. Now $\psi(x) = (A \cap \psi(x))B = (A \cap J)B \subseteq J$.  It follows that $Q = \psi(x)$ or $Q = \psi(y).$ In the first case, $P = \phi(x)$ and in the second, $P = \phi(y)$, and so $\phi(x) \subsetneq \phi(y)$ is saturated.

%Therefore, \(\psi(x) <_c \psi(y)\) in $B$ necessarily implies \(\phi(x) <_c \phi(y)\) in $A$. 

Now suppose that \(\operatorname{ht}(x)=1\). By our assumptions on \(X\), this forces \(\operatorname{ht}(y)=2\). Thus, \(\psi(x)\) and \(\psi(y)\) are prime ideals in \(B\) with heights one and two respectively. 
%Moreover, since \(x <_c y\), we have \(\psi(x) \subsetneq \psi(y)\), and so \(A\) contains a generating set for \(\psi(y)\). 
If we let \(x=x_i\) for some \(i\in\{1,\ldots,m\}\), then by construction there exists a prime element \(\widetilde{x_i}\) in \(A\) such that \(A\cap\psi(x_i)=\widetilde{x_i}A\). Hence, \(\phi(x)=A\cap\psi(x)=\widetilde{x_i}A\) is a height one prime ideal of \(A\). Since $\phi(x) \subsetneq \phi(y)$, we have that ht$(\phi(y)) \geq 2.$ Furthermore, since \(\widehat{A}=\widehat{B}\), it follows by the going down theorem that \(\operatorname{ht}(\phi(y))=\operatorname{ht}(A\cap\psi(y))\leq 2\). Thus, $\phi(x) \subsetneq \phi(y)$ is saturated.

%If \(\operatorname{ht}(\phi(y))=2\), then \(\phi(x) <_c \phi(y)\). Alternatively, if \(\operatorname{ht}(\phi(y))=1\), then the relation \(\psi(x) <_c \psi(y)\) in \(B\) necessarily implies \(x_iA=A\cap\psi(x)=A\cap\psi(y)=\phi(y)\). Given that \(A\) contains a generating set for \(\psi(y)\), it follows that \((x_iA)B=\psi(y)\), which contradicts the fact that \(\psi(y)\) is a height two prime ideal in \(B\). 

Finally, if ht$(x) = 0$, then $x$ is the minimal node of $X$ and $y = x_i$ for some $i = 1,2, \ldots ,m$. In this case, $\psi(x) = (0)$ and so $\phi(x) = (0)$ and $\phi(y) = \widetilde{x_i}A$ is a height one prime ideal of $A$. Thus, \(\phi(x) \subsetneq \phi(y)\) is saturated and it follows that $\phi$ is a saturated embedding. 

To establish that $\phi$ is dimension preserving, we have
$$\dim(X)=\dim(B)=\dim(\widehat{B})=\dim(\widehat{A})=\dim(A)$$
and so $\phi$ is a dimension-preserving saturated embedding.
\end{proof}

\section{Characterization of Dimension-Preserving Saturated Embeddings of UFDs}\label{characterization of dimension-preserving saturated embeddings of UFDs}
In this section, we state and prove the main result of this paper, along with two corollaries.
\begin{theorem}\label{main theorem of characterization of embeddings of UFDs}
    Let $X$ be a finite poset. Then, there exists a local UFD $(A,M)$ and a dimension-preserving saturated embedding $\phi: X \longrightarrow \spec(A)$ if and only if:
    \begin{enumerate}
        \item $X$ has a unique minimal node and unique maximal node, and
        \item If $\dim(X) \geq 2$ and $x, y  \in X$ with $\height(x)=1$ and $x<_c y,$ then $\height(y)=2$.
    \end{enumerate}
\end{theorem}
\begin{proof}
    Suppose that there exists a local UFD $A$ and a dimension-preserving saturated embedding $\phi: X \longrightarrow \spec(A)$. The first condition immediately follows from $A$ being a local integral domain and $\dim(X) = \dim(A)$. For the second condition, suppose that $x, y  \in X$ with $\height(x)=1$ and $x<_c y$. Since $\phi$ is a dimension-preserving saturated embedding, ht$(\phi(x)) = 1$ and $\phi(x)\subsetneq \phi(y)$ is saturated.

Let \( a \in \phi(y) \setminus \phi(x) \) and consider the localized ring \( A' = A_{\phi(y)} \). As \( A' \) is a UFD, the ideal \( \phi(x) A' \) is principal. Now, consider the ideal \( I = (\phi(x), a) A' \), which is generated by two elements. Since the containment \( \phi(x) A' \subsetneq \phi(y) A' \) is saturated, it follows that $\phi(y) A'$ is a minimal prime ideal over $IA'$. By the generalized principal ideal theorem, we obtain \( \height(\phi(y) A') = \height(\phi(y)) = 2 \) and it follows that ht$(y) = 2$. 

Now, suppose that $X$ is a finite poset with a unique minimal node and a unique maximal node, and that whenever $x,y \in X$ satisfy $\height(x)=1$ with $x<_c y$, we have $\height(y)=2$. We consider the following three cases: If \( \dim(X) = 0 \), then the map \( \phi: X \longrightarrow \operatorname{Spec}(\mathbb{Q}) \) defined by \( \phi(x) = (0) \) is a dimension-preserving saturated embedding. If \( \dim(X) = 1 \), consider the ring \( A = \mathbb{Q}[[t]] \) for an indeterminate \( t \), and define the map \( \phi: X \longrightarrow \operatorname{Spec}(\mathbb{Q}[[t]]) \) by setting \( \phi(x) = (0) \) if \( x \) is the minimal node of $X$ and \( \phi(x) = (t) \) if \( x \) is the maximal node of $X$. By construction, this map provides a dimension-preserving saturated embedding. Finally, if \( \dim(X) \geq 2 \), Theorem \(\ref{UFD Theorem together with a saturated embedding}\) guarantees the existence of a local UFD \( A \) and a dimension-preserving saturated embedding \( \phi: X \longrightarrow \operatorname{Spec}(A) \).
\end{proof}

Our conditions in Theorem \ref{main theorem of characterization of embeddings of UFDs} are remarkably mild, demonstrating that, for most finite posets with a unique maximal node and a unique minimal node, there is a local UFD $A$ and a dimension-preserving saturated embedding from the poset to the prime spectrum of $A$. We now present a class of posets that satisfies the first condition of Theorem \ref{main theorem of characterization of embeddings of UFDs}, but does not satisfy the second.

\begin{example}
     The poset $X$ depicted in Figure \ref{DrawingOfX} has exactly one maximal node and one minimal node.  However, the maximal node has height four and covers a node with height one. Thus, $X$ does not satisfy the second condition of Theorem \ref{main theorem of characterization of embeddings of UFDs}. It follows that there does not exist a local UFD $A$ and a dimension-preserving saturated embedding from $X$ to $A$. In fact, a slight modification of $X$, as shown in Figure \ref{DrawingOfModificationX}, reveals a broad class of posets that satisfy the first condition of Theorem \ref{main theorem of characterization of embeddings of UFDs}, but not the second. Thus, there is no dimension-preserving saturated embedding of one of these posets into the prime spectrum of a local UFD.
\end{example}
\begin{figure}
    \centering
    \begin{subfigure}[t]{0.45\textwidth}
        \centering
        \begin{tikzpicture}[
            main/.style={circle, draw, fill=black, inner sep=2pt},
            node distance=1.2cm
        ]
            % NODES
            \node[main] (6) {};
            \node[main] (8) [below right of=6] {};
            \node[main] (9) [below of=8] {};
            \node[main] (12) [below of=9] {};
            \node[main] (10) [below left of=12] {};
            \node[main] (7) [below left of=8, left of=8] {};

            % EDGES
            \draw (6) -- (8)  -- (9) -- (12) -- (10);
            \draw (6) -- (7) -- (10);
        \end{tikzpicture}
        \caption{Poset $X$}
        \label{DrawingOfX}
    \end{subfigure}
    \hfill
    \begin{subfigure}[t]{0.45\textwidth}
        \centering
        \begin{tikzpicture}[
            main/.style={circle, draw, fill=black, inner sep=2pt},
            node distance=0.9cm
        ]
            % NODES
            \node[main] (6) {};
            \node[main] (8) [below right of=6] {};
            \node[main] (9) [below of=8] {};
            \node[main] (12) [below of=9] {};
            \node[main] (13) [below of=12] {};
            \node[main] (10) [below left of=13] {};
            \node[main] (7) [below left of=8, left of=8] {};

            % EDGES
            \draw (6) -- (8) --(9);
            \draw  (12) -- (13) -- (10);
            \draw (6) -- (7) -- (10);

            % Split edge with \vdots
            \path (9) -- (12) coordinate[pos=0.2] (A);
            \path (9) -- (12) coordinate[pos=0.8] (B);
            \draw (9) -- (A);
            \draw (B) -- (12);
            \node at ($(9)!0.4!(12)$) {\scalebox{0.7}{$\vdots$}};
        \end{tikzpicture}
        \caption{Modification of poset $X$}
        \label{DrawingOfModificationX}
    \end{subfigure}
    \caption{Examples of posets that satisy the first condition of Theorem \ref{main theorem of characterization of embeddings of UFDs} but not the second}
    \label{fig:two-posets}
\end{figure}
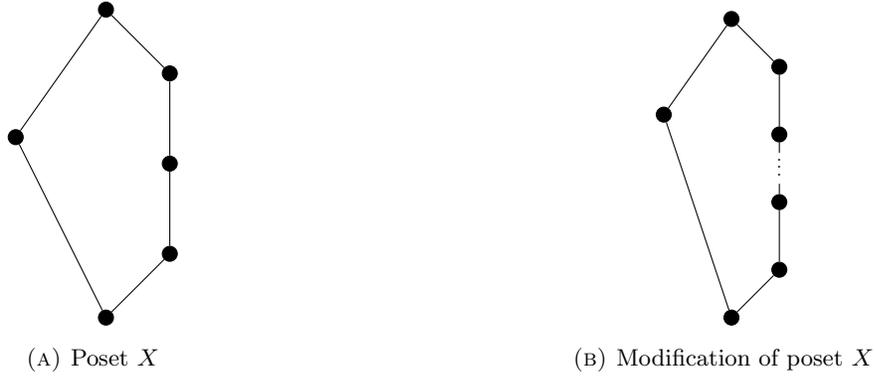

\begin{corollary}\label{corollary UFD for non local posets}
     Let $X$ be a finite poset. Then, there exists a Noetherian UFD $A$ and a dimension-preserving saturated embedding $\phi: X \longrightarrow \spec(A)$ if and only if:
    \begin{enumerate}
        \item $X$ has a unique minimal node, and
        \item If $\dim(X) \geq 2$ and $x, y  \in X$ with $\height(x)=1$ and $x<_c y,$ then $\height(y)=2$.
    \end{enumerate}
\end{corollary}
\begin{proof}
The first condition is necessary since $A$ is a domain and $\phi$ is a dimension-preserving saturated embedding.   The argument for the necessity of the second condition is identical to that given in the proof of Theorem \ref{main theorem of characterization of embeddings of UFDs}. 

To show that the two conditions are sufficient, we follow the proof of Corollary \ref{corollary quasi-excellent 2}, replacing the appeal to Theorem \ref{local quasi-excellent ring codimension preserving theorem} by Theorem \ref{main theorem of characterization of embeddings of UFDs}. Assume that the two stated conditions hold. We define a new poset $X' = X \cup \{z\}$ where the order is given by $y <_{X'} z$ for $y \in X$ and $x \leq_{X'} y$ if and only if $x \leq_X y$ for $x,y \in X$. 
Then $X'$ has a unique maximal node and a unique minimal node, and so by Theorem \ref{main theorem of characterization of embeddings of UFDs} there exists a local UFD $(A',M')$ and a dimension-preserving saturated embedding $\phi' : X' \longrightarrow \Spec(A')$. In particular, dim$(X') = \mbox{dim}(A')$.
Let \(y_1,\dots,y_n\) be the maximal elements of \(X\).  
If $y \in X'$ and $y <_c z$, then $y = y_i$ for some $i = 1,2, \ldots ,n$.  Since $\phi'$ is a saturated embedding, $\phi'(y_i) \subsetneq \phi'(z) = M'$ is saturated for every $i = 1,2, \ldots ,n$.

Set \(\mathfrak S = A'\setminus\bigl(\phi'(y_1)\cup\cdots\cup\phi'(y_n)\bigr)\) and let \(A=\mathfrak S^{-1}A'\).  Then $A$ is a Noetherian UFD.
Moreover, there is an inclusion preserving one-to-one correspondence between the prime ideals of \(A\) and the prime ideals of \(A'\) contained in \(\phi'(y_i)\) for some $i = 1,2, \ldots ,n$. Thus, \(\phi\colon X\to \spec(A)\) defined by $\phi(y)=\phi'(y)A$ is a saturated embedding. Note that dim$(A) \leq \mbox{dim}(A') - 1$.

Let dim$(X') = m$, and note that dim$(X) = m - 1$. Then there is a chain in $X'$ of the form
$$x = x_0 <_c x_1 <_c \cdots <_c x_{m - 1} <_c x_m = z$$ where $x$ is the minimal node of $X'$. Now, $x_{m - 1} = y_i$ for some $i = 1,2, \ldots ,n$. Since $\phi'$ is a saturated embedding, we have a saturated chain of prime ideals
$$(0) = \phi'(x) \subsetneq \phi'(x_1) \subseteq \cdots \subsetneq \phi'(y_i) \subsetneq M'$$ of length $m$ in $A'$, and so we have a corresponding saturated chain of prime ideals
$$(0) = \phi'(x)A \subsetneq \phi'(x_1)A \subsetneq \cdots \subsetneq \phi'(y_i)A$$ in $A$ of length $m - 1$. It follows that $\dim(A) \geq m - 1$, and so we have $\dim(A) = m - 1$. Thus,
$\dim(A) = \dim(X)$. 
Therefore, $\phi$ is a dimension-preserving saturated embedding.
\end{proof}

\begin{corollary}\label{dimplusone}
    Let $X$ be a finite poset. Then, there exists a Noetherian UFD $A$ and a saturated embedding $\phi: X \longrightarrow \spec(A)$ such that $\dim(A)=\dim(X)+1$.
\end{corollary}
\begin{proof}
To begin, we consider the set $$Z = \{x \in X \mid \mbox{ht}(x) = 0 \mbox{ and there exists } y \in X \mbox{ with } x <_c y \mbox{ and ht} (y) \geq 2\}.$$
If $Z$ is empty, we let $X' = X$.  If $Z$ is not empty, let $Z = \{x_1, \ldots ,x_m\}$. Define $$Y = \{(x_1,x_1), (x_2,x_2), \ldots ,(x_m,x_m)\}$$ and define $X' = X \cup Y.$ We now define an order relation on $X'$. Suppose $u,v \in X'$. If $u,v \in X$, then $u \leq_{X'} v$ if and only if $u \leq_{X} v$. If $u,v \in Y$, then define $u \leq_{X'} v$ if and only if $u = v$. Otherwise, $u <_{X'} v$ if and only if $u = (x_i,x_i)$ for some $i = 1,2, \ldots ,m$ and $x_i \leq_X v$. Note that dim$(X') = \mbox{dim}(X)$.

Now define $X'' = X' \cup \{z\}$ where the order on $X''$ is given by $z <_{X''} x$ for $x \in X'$ and $x <_{X''} y$ if and only if $x <_{X'} y$ for $x,y \in X'$. Note that dim$(X'') = \dim(X') + 1 = \dim(X) + 1$.
   
 The conditions of Corollary $\ref{corollary UFD for non local posets}$ are now satisfied for $X''$, so there is a Noetherian UFD \((A,M)\) and a saturated embedding $\phi\colon X''\longrightarrow\spec(A)$ with \(\dim A=\dim X''=\dim X+1\).  Restricting \(\phi\) to \(X\) gives the desired saturated embedding of \(X\) into \(\spec(A)\).\end{proof}

Let $X = \{x,y\}$ where $x$ and $y$ are not comparable. Then $\dim(X) = 0$ and indeed it is possible to find a Noetherian UFD $A$ of dimension one such that there is a saturated embedding $\phi : X \longrightarrow \Spec(A)$. For example, let $S = \mathbb{Q}[[z,w]]$ and let $A$ be $S$ with all elements not in $(z) \cup (w)$ inverted. Then $A$ is a Noetherian UFD of dimension one and the map sending $x$ to $zA$ and $y$ to $wA$ is a saturated embedding.
However, there does not exist a Noetherian UFD $A$ (or even a domain $A$) and a saturated embedding $\phi: X \longrightarrow \Spec(A)$ such that dim$(A) = \dim(X)$. This example shows that the dimension equality in the statement of Corollary \ref{dimplusone} cannot be changed to $\dim(X) = \dim(A)$, and so our dimension equality is the best possible.

\section*{Acknowledgements}

We are extremely grateful to Cory Colbert for many helpful conversations, comments, and suggestions related to the results in this paper.

\end{document}